\documentclass{amsart}

\usepackage{amsthm}
\usepackage{amssymb}
\usepackage{amsmath}
\usepackage{xfrac}
\usepackage{array}
\usepackage{mathtools}
\usepackage{stmaryrd}
\usepackage{bbm}
\usepackage[multiple]{footmisc}

\newtheorem{prop}{Proposition}
\newtheorem{definition}[prop]{Definition}
\newtheorem{theorem}[prop]{Theorem}

\newtheorem{thm}[prop]{Theorem}
\newtheorem{cor}[prop]{Corollary}
\newtheorem{lemma}[prop]{Lemma} 
\newtheorem{remark}[prop]{Remark}

\newtheorem{exam}[prop]{Example}
\numberwithin{prop}{section}

\DeclarePairedDelimiter\abs{\lvert}{\rvert}
\DeclarePairedDelimiter\norm{\lVert}{\rVert}

\newcommand{\RR}{\ensuremath{\mathbb{R}}}

\newcommand{\CC}{\ensuremath{\mathbb{C}}}

\newcommand{\GG}{\ensuremath{\mathbb{G}}}

\newcommand{\NN}{\ensuremath{\mathbb{N}}}

\newcommand{\QQ}{\ensuremath{\mathbb{Q}}}
\newcommand{\Qp}{\ensuremath{\mathbb{Q}_{p}}}

\newcommand{\Fq}{\ensuremath{\mathbb{F}_{q}}}
\newcommand{\Fqt}{\ensuremath{\Fq[[t]]}}
\newcommand{\Fqtt}{\ensuremath{\Fq\left(\left(t\right)\right)}}
\newcommand{\Zp}{\ensuremath{\mathbb{Z}_{p}}}

\newcommand{\OO}{\ensuremath{\mathcal{O}}}

\newcommand{\CAT}{\ensuremath{\mathrm{CAT}(0)}}

\newcommand{\IRS}[1]{\ensuremath{\mathrm{IRS}\left( #1 \right) }}

\newcommand{\Sub}[1]{\ensuremath{\mathrm{Sub}\left( #1 \right) }}
\newcommand{\Cl}[1]{\ensuremath{\mathrm{Cl}\left( #1 \right) }}

\newcommand{\Gr}[1]{\ensuremath{\mathcal{GR}(#1)}}

\newcommand{\nrm}{\ensuremath{\vartriangleleft }}

\begin{document}

\title{Invariant random subgroups over non-Archimedean local fields}
\author{Tsachik Gelander and Arie Levit}

\begin{abstract}

Let $G$ be a higher rank semisimple linear algebraic group over a non--Archimedean local field.  
The simplicial complexes corresponding to any sequence of pairwise non-conjugate irreducible lattices in $G$ are Benjamini--Schramm convergent to the Bruhat--Tits building. 
Convergence of the relative Plancherel measures and normalized Betti numbers follows. This extends the work  \cite{7S} from real Lie groups to linear groups over arbitrary local fields. 
Along the way, various results concerning Invariant Random Subgroups and in particular a variant of the classical Borel density theorem are also extended.
\end{abstract}

\maketitle

\setcounter{tocdepth}{1}
\tableofcontents

\section{Introduction}
\label{sec:introduction}

We study lattices in semisimple analytic groups. The latter are roughly speaking direct products of simple linear groups over arbitrary local fields.
Our goal is to establish certain asymptotic properties holding true for \emph{any} sequence of pairwise non-conjugate irreducible lattices. The following statement is our main result.

\begin{theorem}
	\label{thm:accumulation points of invariant random subgroups}
	Let $G$ be a semisimple analytic group. 
	Assume that $G$ is happy, has property $(T)$ and  $\mathrm{rank}(G) \ge 2$.
	Let $\Gamma_n$ be a sequence of pairwise non-conjugate irreducible lattices in $G$.  Then any accumulation point of the  corresponding invariant random subgroups $\mu_{\Gamma_n}$ is equal to $\delta_Z$ for some central subgroup $Z$ in $G$.
\end{theorem}

A semisimple analytic group is \emph{happy} if its non--Archimedean factor contains a topologically finitely generated compact open subgroup.  This notion  is motivated by \cite{barnea} and discussed in \S\ref{sub:G+ and happy} below. All semisimple analytic groups zero are happy in zero characteristic, and in positive characteristic all simply-connected semisimple analytic group are happy.

In positive characteristic the $\Gamma_n$'s are required to be pairwise non-conjugate by automorphisms of $G$. Of course it is enough to assume that $\text{Vol}(G/\Gamma_n)\to \infty$.

%
%


The Archimedean case of Theorem \ref{thm:accumulation points of invariant random subgroups} has been established in \cite{7S} and our approach is basically the same as the one developed there. In particular 
the key strategy used in \S\ref{sec:accumulation points of invariant random subgroups}  towards our main theorem  is similar to that of \cite[\S 4.4]{7S}. 
However we are forced  to deal with many issues that do not come up in the Archimedean case. 

One of the difficulties encountered in the setting of arbitrary local fields has to do with positive characteristic phenomena, to which we pay particular attention. $p$-adic Lie groups are on the other hand easier to deal with than their real counterparts, and the mixed case usually requires little extra effort.

The property $(T)$ assumption in Theorem \ref{thm:accumulation points of invariant random subgroups} is needed  to  rely on the deep theorem of Stuck and Zimmer. 
Let us note that, unlike the situation in \cite{7S}, this is the only reason that property $(T)$ is essential for our approach. While it is used also in \S \ref{sec:accumulation points of invariant random subgroups} when referring to \cite{glasner1997kazhdan}, this is done only to simplify the argument.

The property of accumulation points in question admits a geometric reformulation in terms of Benjamini--Schramm convergence. This and the notion of invariant random subgroups
are explored in \S\ref{sec:Chabauty and BS convergence}  and  used throughout this work.



\subsection{Relative Plancherel measure convergence}

One of the major achievements of \cite{7S} is a convergence formula for relative Plancherel measures with respect to a general uniformly discrete sequence of lattices. 
To be precise, let $\Gamma$ be a uniform lattice in $G$ so that the right quasi-regular representation $\rho_\Gamma$ of $G$ in $L^2(\Gamma \backslash G,\mu_G)$  decomposes as a direct sum of irreducible representations.
Every  irreducible unitary representation $\pi$ of $G$  appears in $\rho_\Gamma$ with  finite multiplicity $m(\pi,\Gamma)$. 

\begin{definition}
\label{def:relative plancherel}
The \emph{normalized relative Plancherel measure} of $G$ with respect to $\Gamma$ is an atomic measure on the unitary dual $\widehat{G}$  given by
	$$
	\nu_\Gamma = \frac{1}{\mathrm{vol}(\Gamma \backslash G)} \sum_{\pi \in \widehat{G}} m(\pi, \Gamma) \delta_\pi.
	$$
\end{definition}

It turns out that the following deep representation-theoretic statement follows from the geometric asymptotic property of the kind established in Theorem \ref{thm:accumulation points of invariant random subgroups}.


 

\begin{theorem}
	\label{thm:BS convergence implies plancherel convergence}
Let $G$ be a  semisimple analytic group in zero characteristic. Fix a Haar measure on $G$ and let $\nu_G$ be the associated Plancherel measure on $\widehat{G}$.
Let $\Gamma_n$ be a uniformly discrete  sequence of lattices in $G$ with $\mu_{\Gamma_n}$ being weak-$*$ convergent to $\delta_{\{e\}}$.  Then 
	$$ \nu_{\Gamma_n}(E) \xrightarrow{n \to \infty} \nu_G(E) $$
	for every relatively quasi-compact  $\nu^G$-regular  subset $E \subset \widehat{G}$.
\end{theorem}

The proof of Theorem \ref{thm:BS convergence implies plancherel convergence} is essentially the same as \cite[1.2,6.7]{7S}. It relies on the Plancherel formula and the Sauvageot principle; see \S\ref{sec:plancherel}. To the best of our knowledge, the latter principle appears in the literature only in zero characteristic, and for that reason we exclude positive characteristic local fields in Theorem \ref{thm:BS convergence implies plancherel convergence}.

Let $d_\pi$ denote the \emph{formal dimension} of the irreducible unitary representation $\pi$  in the regular representation of $G$. 

\begin{cor}
\label{cor:BS convergence implies ptwise plancherel convergence}
In the situation of Theorem \ref{thm:BS convergence implies plancherel convergence} we have  that moreover
$$ \frac{m(\pi,\Gamma_n)}{\text{vol}(\Gamma_n\backslash G)} \xrightarrow{n \to \infty} d_\pi$$
for every $\pi \in \widehat{G}$.
\end{cor}

Putting together the two Theorems \ref{thm:accumulation points of invariant random subgroups} and \ref{thm:BS convergence implies plancherel convergence} immediately gives the following.


\begin{cor}
\label{cor:Plancharel convergence for lattices in higher rank T}
Let $G$ be a  semisimple analytic group in zero characteristic. Assume that $G$ has property $(T)$ and that $\mathrm{rank}(G) \ge 2$.  Then the conclusions of Theorem \ref{thm:BS convergence implies plancherel convergence} and Corollary \ref{cor:BS convergence implies ptwise plancherel convergence} apply to any sequence of uniformly discrete, torsion free, pairwise non-conjugate irreducible lattices.
\end{cor}

\begin{remark}
	Recall that in the non-Archimedean zero characteristic case any sequence of lattices in uniformly discrete. If $G$ is center-free then there is no need to assume torsion-freeness.
\end{remark}

\subsection{Normalized Betti numbers}
Another achievement of \cite{7S} is a convergence formula for normalized Betti numbers, derived as a consequence of the convergence of relative Plancherel measures. 
We present here a direct argument applicable in the more combinatorial non-Archimedean case, relying on a beautiful result of Elek and avoiding Theorem \ref{thm:BS convergence implies plancherel convergence}. This holds true in positive characteristic as well. 

Recall that $b_d^{(2)}(G)$ denote the  $L^2$-Betti numbers of the group $G$, c.f.  \cite{petersen2013l2}.


\begin{theorem}
\label{thm:limit of normalized betti numbers is l2-betti number}
Let $G$ be a non-Archimedean semisimple analytic group.
Let $\Gamma_n$ be any sequence of torsion-free uniform lattices in $G$ with $\mu_{\Gamma_n}$ being weak-$*$ convergent to $\delta_{\{e\}}$.  Then
$$ 
 \lim_{n\to\infty} \frac{b_d(\Gamma_n)}{\text{vol}(G/\Gamma_n)} 
 = b_d^{(2)}(G)
$$
for all $d \in  \NN \cup \{0\}$.
\end{theorem}

Theorem \ref{thm:limit of normalized betti numbers is l2-betti number} has been established for uniformly discrete sequences of lattices in the general framework of (not necessarily algebraic) locally compact totally disconnected groups in a recent interesting   work by Petersen, Sauer and Thom \cite{petersen}. The approach of \cite{petersen} applies, to some extent, also to non-uniformly discrete families of lattices, in which case the $L^2$-Betti number gives a lower bound on the normalized Betti numbers. This includes in particular families of non-uniform lattices.
 
The following result is an immediate corollary of Theorems \ref{thm:BS convergence implies plancherel convergence} and \ref{thm:limit of normalized betti numbers is l2-betti number}.

\begin{cor}
\label{cor:betti number convergence for property T higher rank}
Let $G$ be a happy non-Archimedean semisimple analytic group with  $\mathrm{rank}(G) \ge 2$ and property $(T)$. Let $X$ be the Bruhat--Tits building of $G$. Let $\Gamma_n$ be a uniformly discrete\footnote{The uniform discreteness assumption is redundant in characteristic $0$.} sequence of pairwise non-isomorphic, torsion-free, irreducible lattices in $G$. Then
$$ \lim_{n\to\infty}\frac{b_d(\Gamma_n)}{\left|V(\Gamma_n \backslash X)\right|} = b_d^{(2)}(X)$$
for all $ d \in \NN \cup \{0\}$. 
\end{cor}



\subsection{Borel density theorem for invariant random subgroups}

Recall that an invariant random subgroup of $G$ is a  Borel probability measure on the Chabauty space of closed subgroups which is invariant under conjugation. This generalizes the classical notion of lattices, and is used in the proof of our main result.

We extend two well-known and classical results to invariant random subgroups.  The first is the Borel density theorem, saying that a lattice $\Gamma$ in a semisimple analytic group $G$ is Zariski dense provided that  $G$ has no compact factors. The other one concerns closed subgroups $H \le G$ such that $G/H$ admits finite $G$-invariant measure. For example, if $G$ is simple then such $H$ must be discrete. In zero characteristic the second result follows immediately from the first, but it requires an additional argument in positive characteristic ---   see  \cite[II.5]{Ma} or the more recent \cite{pink2004weil}.

\begin{theorem}[Borel density theorem for IRS]
\label{thm:borel density for IRS}
Let $k$ be a local field and $G$ a happy semisimple analytic group over $k$. Assume that $G$ has no  almost $k$-simple factors of type $B_n,C_n$ or $F_4$ if $\mathrm{char}(k) = 2$ and of type $G_2$ if $\mathrm{char}(k) = 3$. 


Let $\mu$ be an  ergodic invariant random subgroup of $G$. Then there is a pair of  normal subgroups $N,M \nrm G$ so that 
 $$ N \le H \le M,  \quad  \text{$H/N$ is discrete in $G/N$} \quad \text{and} \quad \overline{H}^{\mathrm{Z}} = M $$
 for $\mu$-almost every closed subgroup $H$ in $  G$.  Here $\overline{H}^{\mathrm{Z}}$ is the Zariski closure of $H$. 
\end{theorem}

A variant of Theorem \ref{thm:borel density for IRS} for simple real Lie groups first appeared in \cite{7S}. We use related ideas coming from Fursenberg's proof of the classical Borel density theorem in \cite{furst_borel}. We would like to mention that geometric and Zariski density of invariant random subgroups was recently considered in the work \cite{duchesne2015geometric}.

One major difficulty peculiar to positive characteristic is the lack of a useful correspondence between closed subgroups and Lie subalgebras.  Moreover, the $\mathrm{Ad}$-representation of a simple linear group need no longer be irreducible. We overcome these issues relying on a celebrated criterion by Pink \cite{pink} as well as a certain linearization technique developed in \S\ref{sec:zariski closure and lie algebras}.

While we believe Theorem \ref{thm:borel density for IRS} to be of independent interest, it is not used towards proving other results in this work. 

\subsection{Lattices in the Chabauty space}

Let $G$ be a semisimple analytic group and $\Gamma$ an irreducible lattice in $G$. Our approach towards Theorem \ref{thm:accumulation points of invariant random subgroups} suggests that we   study  small neighborhoods of $\Gamma$ in the Chabauty space associated to $G$. This shares a common theme with a previous work of the authors \cite{GL} where an analogue question was resolved for uniform lattices\footnote{Uniform lattices in isometry groups of proper, geodesically complete $\CAT$ spaces are considered in  \cite{GL}. This setup suffices for the non-Archimedean zero characteristic  case.}. In particular, the following theorem can be seen as an extension of \cite{GL} to the higher rank non-uniform case.

%

%
%

\begin{theorem}[Chabauty local rigidity]
\label{thm:Chabauty local rigidity for lattices in algebraic groups}
If $\mathrm{rank}(G) \ge 2$ then the irreducible lattice $\Gamma$ admits a Chabauty neighborhood consisting of  conjugates of $\Gamma$ by automorphisms of $G$. 
If  $G$ is  defined over zero characteristic then these automorphisms are inner.
\end{theorem}

Theorem \ref{thm:Chabauty local rigidity for lattices in algebraic groups} is essentially a variant of local rigidity, expressed in terms of small deformations in the Chabauty topology rather than the representation space of $\Gamma$ in $G$ which more algebraic in nature.

The following immediate corollary is to be compared with a closely related statement for lattices in semisimple Lie groups with property $(T)$ established in 
\cite[4.6]{7S}.  In particular it holds true for any sequence of irreducible lattices which are pairwise non-isomorphic or with covolume  tending to infinity. Let us reiterate  that our current approach in this matter is independent of property $(T)$.




\begin{cor}
\label{cor:non conjugate lattices are discrete IRS}
If $\mathrm{rank}(G) \ge 2$ and $\Gamma_n$ is a sequence of irreducible lattices in $G$ pairwise non-conjugate by automorphisms of $G$ then the sequence $\mu_{\Gamma_n}$ is discrete in the space of extremal points of $\IRS{G}$.
If $G$ is defined over zero characteristic then it suffices to assume that the $\Gamma_n$ are pairwise non-conjugate.
\end{cor}



Chabauty local rigidity can be used to deduce Wang's finiteness theorem --- 

\begin{theorem}[Wang's finiteness]
\label{thm:Wang}
Assume that  $\mathrm{rank}(G) \ge 2$, and if $G$ is defined over positive characteristic then assume that $G$ is simply connected. 

Then $G$ admits only finitely many $\mathrm{Aut}(G)$-conjugacy classes of irreducible lattices with co-volume bounded by any fixed $ v > 0$. 
In the zero characteristic case $G$ admits finitely many \emph{inner} conjugacy classes as above.
\end{theorem}

Theorem \ref{thm:Wang} is an extension of \cite[Theorem 1.5]{GL} relying on Chabauty local rigidity as in our Theorem \ref{thm:Chabauty local rigidity for lattices in algebraic groups} and  therefore  applying to non-uniform lattices as well. This formulation of Wang's finiteness follows from the well-known results of Borel and Prasad \cite{bp}. The present approach is however potentially easier. 

We remark that Theorem \ref{thm:Wang} complements Corollary \ref{cor:non conjugate lattices are discrete IRS} in the sense that any sequence of irreducible lattices considered in the latter corollary must have co-volume tending to infinity.


\section{Semisimple linear groups over local fields}
\label{sec:semisimple linear groups over local fields}

For the reader's convenience we present and summarize well-known material used in this work. We define and discuss semisimple analytic groups. Several key results regarding these groups are mentioned,  in particular a deep result of Pink \cite{pink}.

\label{sub:notations and assumption}

%


\subsection{Local fields}
An \emph{Archimedean local field} is either $\RR$ or $\CC$.

A  \emph{non-Archimedean local field} is  a finite extension  either of  $\Qp$  or of the field $\Fqtt$ of formal Laurent series over $\mathbb{F}_q$, where  $\mathbb{F}_q$ is the finite field with $q$ elements and $q=p^n$ is a prime power. Note that $\mathrm{char}(\Qp) = 0 $ and $\mathrm{char}(\Fqtt) = p$. 
The \emph{prime field} of $k$ is $\QQ$ in the first case and $\mathbb{F}_p$ in the second case. Let $\OO$ denote the \emph{valuation ring} of $k$ and $\pi$ a \emph{uniformizer element}. 
For example, if $k=\Qp$ then $\OO = \Zp$ and one may take $ \pi = p$, and if $k=\Fqtt$, $\OO = \Fqt$ and one may take $ \pi = t$. 

\begin{remark}
In a few arguments below our main efforts are intended to deal with the positive characteristic case, the zero characteristic case being easier.
\end{remark}

\subsection{Semisimple analytic groups}

Our main objects of study are semisimple analytic groups.

\begin{definition}
\label{def:semisimple analytic group}
Let $k$ be a  local field and $\GG$ a connected  $k$-isotropic $k$-simple  linear $k$-algebraic group.
\begin{itemize}
\item A \emph{simple analytic group} is a group of the form $\GG(k)$.
\item A \emph{semisimple analytic group} is an almost direct product of finitely many simple analytic groups, possibly over different local fields.
\end{itemize}
\end{definition}

Note that if $k$ is a local field and $\GG$ is a connected semisimple linear $k$-algebraic group without $k$-anisotropic factors then $\GG(k)$ is a semisimple analytic group. Such a group is indeed \emph{analytic} in the sense of e.g. \cite{sere}.

A semisimple analytic group is respectively \emph{non-Archimedean/has zero characteristic/simply connected} if all of its simple analytic almost direct factors are non-Archimedean/defined over a local field of zero characteristic/are simply connected.

Associated to a  semisimple analytic group $G$ are its universal covering group $\widetilde{G}$ and adjoint group $\overline{G}$. There are central $k$-isogenies $\widetilde{G} \xrightarrow{\widetilde{p}} G \xrightarrow{\overline{p}} \overline{G} $ and this data is unique up to a $k$-isomorphism \cite[I.4.11]{Ma}.


%

\subsection{The subgroup $G^+$ and happy  semisimple analytic groups} 
\label{sub:G+ and happy}

Let $G$ be a semisimple analytic group. Denote by $G^+$ the subgroup  of $G$ generated by its unipotent elements. This notion is rather important for our needs, especially when  $G$ is not simply connected. We summarize  some of the properties of $G^+$ and refer to \cite[I.1.5,I.2.3]{Ma} and the references therein for more details. 

If $G$ is simply connected then $G = G^+$. If $G$ is Archimedean then $G^+$ is the connected component $G_0$ at the identity. In general $G/G^+$ is a compact abelian group. The group $G^+$ admits no proper finite index subgroups.  

The group $G^+$ is useful in the description of normal subgroups ---
\begin{prop}
\label{prop:normal subgroups of $G$}
Let $N$ be a subgroup of $G$ normalized by $G^+$. 
Then $N$ is normal in $G$ and there is an almost direct factor $M$ of $G$ such that $M^+ \le N$ and $NM/M$ is central in $G/M$.
 \end{prop}
\begin{proof}
Let $H$ be any almost $k$-simple factor of $G$. The projection of $N$ to $H$ as well as the intersection of $N$ with $H$ are normalized by $H^+$. Therefore these two subgroups of $H$ are either central or contain $H^+$ \cite[I.1.5.6]{Ma}. The proposition follows from these facts and Goursat's lemma in elementary group theory.
\end{proof}

We introduce a certain notion of happiness that implies that the subgroup $G^+$ is particularly nicely behaved and will be useful  in dealing with non-simply-connected semisimple analytic groups.

\begin{definition}
\label{def:happy semisimple analytic group}
A simple analytic group $G$ is \emph{happy} if $\mathrm{char}(k)$ does not the divide $\abs{Z}$ where $Z$ is the kernel of the map $\widetilde{G} \to G$. A semisimple analytic group is \emph{happy} if all of its almost   direct factors are.
\end{definition}

Note that a simply connected or a zero characteristic semisimple analytic group is automatically happy.


\begin{thm}[Barnea--Larsen]
	\label{thm:properties of happy groups}
	The following are equivalent for a semisimple analytic group $G$. 
	\begin{enumerate}
\item \label{item:G happy} $G$ is happy,
\item \label{item:map is open} The central $k$-isogeny $\widetilde{p} : \widetilde{G}\to G$ is separable, or equivalently is an open map in the Hausdorff topology,
\item \label{item:quotient is finite} $G/G^+$ is a finite abelian group,
\item \label{item:finitely generated compact open} Some (equivalently every) compact open subgroup in the  non-Archimedean factor of $G$ is  finitely generated.
\end{enumerate}
\end{thm}

The work of Barnea and Larsen focuses on property (\ref{item:finitely generated compact open}). Some of the equivalences between (\ref{item:G happy}), (\ref{item:map is open}) and (\ref{item:quotient is finite}) are  discussed already e.g. in  \cite{borel1973homomorphismes}.

\begin{proof}
The Archimedean factor $G_\text{c}$ of $G$ is happy by definition. The map $\widetilde{p}_\text{c} : \widetilde{G}_\text{c} \to G_\text{c}$ is a local diffeomorphism. The subgroup $G^+_\text{c}$ is equal to the connected component $G_\text{c,0}$ at the identity of $G_\text{c}$ and $G_\text{c}/G_\text{c,0}$ is  finite. Therefore (\ref{item:G happy}), (\ref{item:map is open}) and (\ref{item:quotient is finite}) all  hold automatically for the Archimedean factor.

Recall that $\widetilde{p}(\widetilde G) = G^+$ \cite[I.1.5.5]{Ma}. The equivalence of  (\ref{item:G happy}), (\ref{item:map is open}) and (\ref{item:quotient is finite}) for the non-Archimedean factor is established in \cite[4.2]{barnea}, and this is equivalent to (\ref{item:finitely generated compact open}) by \cite[4.5]{barnea}. The fact that $G/G^+$ is abelian is contained in \cite[I.1.5]{Ma}.
\end{proof}

We remark that \cite{barnea} relies on Pink's theorem \cite{pink} discussed below.

\subsection{Pink's theorem} 

We discuss a deep result of Pink, which is in a sense a generalization to positive characteristic of certain results due to Weyl in the real case and Chevalley in the $p$-adic case on compact subgroups of algebraic groups.

Let $k$ be a local field\footnote{The local field $k$ in Pink's theorem is not assumed to be non-Archimedean.} and $G$ a simple analytic group over $k$. The following notation will be useful in making a concise statement of Pink's results.



\begin{definition}
\label{def:trace field}
Let $\rho$ be a $k$-linear representation of $G$. For every subgroup $\Gamma$ of $G$ the \emph{trace field} $\mathcal{TF}_\rho(\Gamma)$ is the closed subfield of $k$ generated by the traces $\mathrm{tr} \rho(g)$ for all $g \in \Gamma$.
\end{definition}

We recall two notions used below. First, let $\mathcal{R}_{k/l}$ denote Weil's restriction of scalars functor from  $k$ to its subfield $l$.  Next, recall from \cite[\S 1]{pink} that the commutator on $G$ factors through a generalized commutator map $\widetilde{\left[\, , \,\right]} :G \times G \to \widetilde{G}$.
%

%

\begin{theorem}[Pink]
\label{thm:pink}
Assume that $G$ is adjoint, absolutely simple\footnote{$G$ is \emph{absolutely simple} if it is simple over the algebraic closure $k_\text{alg}$ of $k$.} and has no non-standard isogenies\footnote{$G$ has no non-standard isogenies unless $\mathrm{char}(p) = 2$ and $G$ is of type $B_n, C_n$ or $F_4$ or $\mathrm{char}(p) = 3$ and $G$ is of type $G_2$. See \cite[\S1]{pink} for more details.}. 

Then $G$ admits a $k$-linear representation $\rho$ such that for every compact Zariski dense subgroup $\Gamma$  the field $k$ is a finite extension of the trace field $\mathcal{TF}_\rho(\Gamma)$. 

If   the closure of  $\widetilde{\left[\Gamma,\Gamma\right]}$ in $\widetilde{G}$ is  moreover Zariski dense in $\mathcal{R}_{k/\mathcal{TF}_\rho(\Gamma)}(\widetilde{G})$ then it  is compact and open in $\widetilde{G}$.
\end{theorem}

The notation $H \otimes k$ below stands for $H$ regarded as a $k$-linear group.

\begin{proof}
The Main Theorem 0.2 of \cite{pink} applies to any compact Zariski dense subgroup $\Gamma$ of $G$. 
It provides
a closed subfield $l \subset k$ such that $k/l$ is a finite extension, an $l$-linear adjoint absolutely simple group $H$ and a $k$-isogeny $\varphi : H \otimes k \to G$. Since $G$ admits no non-standard isogenies the map $\varphi$ is in fact a $k$-isomorphism \cite[1.7]{pink}.
Let $\widetilde{H}$ and $\widetilde{\varphi}$ denote the universal cover of $H$ and the corresponding $k$-isomorphism $\widetilde{\varphi} : \widetilde{H} \otimes k \to \widetilde{G}$, respectively.  Moreover there is a compact open subgroup $O \subset \widetilde{H}(l)$ so that $\widetilde{\varphi}(O)$ is equal to the closure of $\widetilde{ \left[\Gamma,\Gamma \right]}$ in $\widetilde{G}$.

The existence of a suitable $k$-representation $\rho$ of $G$ independent of $\Gamma$ and the fact that $l$ is equal to $\mathcal{TF}_{\rho}(\Gamma)$ follow from \cite[0.6(a)]{pink}. We are using the fact that $\tilde\varphi$ is  a $k$-isomorphism  in our case \cite[p. 503]{pink}.

Assume moreover that $\Gamma$ regarded as a subgroup of the group of $l$-points of $\mathcal{R}_{k/l}(\widetilde{G})$ is Zariski-dense. This implies that  $l$ is in fact equal to $k$ \cite[2.4]{pink2004weil}. 
Therefore $\widetilde{\varphi}(O)$ is a compact open subgroup of $\widetilde{G}$.
\end{proof}

In particular, if the adjoint group $G$ is happy then the second conclusion of Theorem \ref{thm:pink} implies that $\Gamma$ itself is open in $G$. We will not be using this observation.

\begin{prop}
\label{prop:trace field of finite index subgroups}
Assume that $G$ is adjoint and absolutely simple. Let $\Gamma$ be a closed subgroup  and $V \subset U$ a pair of compact open subgroups  in $G$. If $\Gamma \cap V$ is Zariski dense then $\mathcal{TF}_\rho(\Gamma \cap U) = \mathcal{TF}_\rho(\Gamma \cap V)$ where $\rho$ is  as in Theorem \ref{thm:pink}. 
\end{prop}
\begin{proof}
This follows from the uniqueness established in \cite[0.2.(b)]{pink}.
\end{proof}

\begin{remark}
One of the main difficulties  in Pink's theorem is the possibility  that the adjoint representation of $G$ might not be irreducible in positive characteristic. There is a simpler proof provided that $\mathrm{Ad}$ is irreducible, see \cite[0.7]{pink}.

The same problem arises when proving discreteness for invariant random subgroups  in our Theorem \ref{thm:borel density for IRS}, and this was part of the motivation in considering Zariski closure instead of Lie algebras in \S\ref{sec:zariski closure and lie algebras} below.
\end{remark}

\subsection{Maximal compact open subgroups}

If $G$ is a simply connected simple analytic group then  a maximal compact open subgroup of $G$ is maximal. This result is  called Tits' theorem  \cite{prasad1982elementary}. The following is true in the general semisimple case.

\begin{theorem}
Let $G$ be a  semisimple analytic group and $U $ an open subgroup in $G$.  Then there is a normal subgroup $N \nrm G$ so that $N \le U$ and $U/N$ is compact.
\label{thm:maximal compact open is maximal}
\end{theorem}
\begin{proof}
	Since $U$ is open the quotient $G / U$ is discrete. Therefore the discrete metric on $G / U$ is $G$-invariant. The result now follows directly from Theorem 6.1 of \cite{equi}.
\end{proof}

As an  immediate corollary of Theorems \ref{thm:properties of happy groups} and \ref{thm:maximal compact open is maximal} and Proposition \ref{prop:normal subgroups of $G$}   we obtain the following fact.

\begin{cor}
\label{cor:finitely many supergroups}
Let $G$ be a happy semisimple analytic group and $K$ a maximal compact subgroup of $G^+$. Then there are only finitely many intermediate subgroups $H$ satisfying $K \le H \le G$.
\end{cor}


\subsection{Automorphisms of semisimple analytic groups}

Let $G$ be  a semisimple analytic group. The following facts regarding topological automorphisms of $G$ is well known. However, we are not aware of a specific reference in the literature.

Consider the group $\mathrm{Aut}(G)$ of the topological group automorphisms of $G$  with the Braconnier topology \cite{braconnier1945groupes} and the group of outer automorphisms $\mathrm{Out}(G) = \mathrm{Aut}(G) / \mathrm{Inn}(G)$ with the quotient topology.

\begin{prop}\label{prop:Out(G)-compact}
The group of outer automorphisms $\mathrm{Out}(G)$ is compact.
\end{prop}

\begin{proof}
We first show how to reduce the problem to the simply connected case. Recall that $\widetilde{G}$ is the universal covering group of $G$ and $p : \widetilde{G} \to G$ is the corresponding central $k$-isogeny with $\ker p$ being central in $\widetilde{G}$.  Since the center $Z(\widetilde{G})$ is a characteristic subgroup, in effect  $\mathrm{Aut}(\widetilde{G})$ is acting on the finite set $Z(\widetilde{G})$. Let $\mathrm{Aut}_p(\widetilde{G})$ denote the finite index  closed subgroup of $\mathrm{Aut}(\widetilde{G})$ consisting of these automorphisms fixing $\mathrm{ker}(p)$  point-wise. Clearly $\mathrm{Inn}(\widetilde{G}) \le \mathrm{Aut}_{p}(\widetilde{G}) \le \mathrm{Aut}(\widetilde{G})$.

The map $p$ induces a continuous and surjective map $p_* : \mathrm{Aut}_{p}(\widetilde{G}) \to \mathrm{Aut}(G)$. Observe that 
$$
 \mathrm{Inn}(\widetilde{G}) \cong \widetilde{G} / Z(\widetilde{G}), \quad\mathrm{Inn}(G) \cong G / Z(G) \quad \text{ and} \quad  Z(G) \cong Z(\widetilde{G}) / \mathrm{ker}(p)
$$
and $p_*$ restricts to a bijection from $\mathrm{Inn}(\widetilde{G})$ to $\mathrm{Inn}(G)$. In particular $p_*$ induces a surjective continuous map $p_* : \mathrm{Out}_p(\widetilde{G}) \to \mathrm{Out}(G)$. Since $\mathrm{Out}_p(\widetilde{G})$ is closed in $\mathrm{Out}(\widetilde{G})$ it suffices to show that $\mathrm{Out}(\widetilde{G})$ is compact.

The simply connected group  $\widetilde{G}$  is the direct product of its almost $k_i$-simple factors $\widetilde{G}_1, \ldots, \widetilde{G}_n$ defined over various local fields $k_i$ \cite[I.4.10]{Ma}.  $\mathrm{Out}(\widetilde{G})$ admits a finite index subgroup equal to $\prod_{i=1}^n\mathrm{Out}(\widetilde{G}_i)$  preserving each factor. 

The theory of abstract homomorphisms of isotropic algebraic groups \cite{borel1973homomorphismes}, \cite[I.1.8]{Ma} shows that every automorphism of the  topological group $\widetilde{G}_i$ is determined by an algebraic  $k_i$-automorphism and an automorphism of the local field $k_i$. Recall that the group of $k_i$-automorphisms of $\widetilde{G_i}$ is the semi-direct product of the group of inner $k_i$-automorphisms by a finite group \cite[1.5.6]{tits1966classification}. 
We conclude that  $\mathrm{Out}(\widetilde{G})$ is compact.
\end{proof}

\noindent
{\bf Remark:} {\it In zero characteristic the group $\mathrm{Out}(G)$ of outer automorphisms is actually finite.}

\begin{cor}
\label{prop:outer automorphism group preserves Haar measure}
Topological group automorphisms of $G$ preserve Haar measure.
\end{cor}

\begin{proof}
The modular function $\triangle : \mathrm{Aut}(G) \to \RR_{>0}^*$  is given by
$ \triangle\left(\alpha\right) \mu_G = \alpha_* \mu_G $
where $\mu_G$ is a fixed Haar measure on $G$. $\triangle$ is continuous with respect to the Braconnier topology \cite[IV.\S 3.]{braconnier1945groupes}.  Since $G$ is unimodular $\triangle(\mathrm{Inn}(G)) =1 $ and the result follows from Proposition \ref{prop:Out(G)-compact}.
 \end{proof}

\section{The Chabauty topology and Benjamini--Schramm convergence}
\label{sec:Chabauty and BS convergence}

In this section we study the Chabauty topology, invariant random subgroups, spaces of quotients and Benjamini--Schramm convergence and explore the relationship between these various notions. We focus on comparing three closely related properties of lattice sequences --- being weakly trivial, IRS convergence to $\delta_\{e\}$  and Benjamini--Schramm convergence of the quotients.
 
\subsection{The Chabauty topology}

Let $G$ be any second countable locally compact group\footnote{From now on we write l.c.s.c. for locally compact second countable.}.


\begin{definition}
\label{def:Chabauty topology}
The \emph{Chabauty space} $\Sub{G}$ of all closed subgroups of $G$ is equipped with the Chabauty topology generated by the following sub-basis sets
\begin{itemize}
\item $\mathcal{O}_1(K) = \{ \text{$H \le G$ closed } \: : \: H \cap K = \emptyset \} $ for every compact subset $K \subset G$, 
\item $\mathcal{O}_2(U) = \{ \text{$H \le G$ closed } \: : \: H \cap U \neq \emptyset \} $ for every open subset $U \subset G$.
\end{itemize}
\end{definition}

The space $\Sub{G}$ is compact and admits a continuous $G$-action by conjugation. 

\begin{definition}
\label{def:invariant random subgroup}
An \emph{invariant random subgroup} of $G$ is a Borel probability measure on $\Sub{G}$ which is invariant under conjugation by $G$. 
\end{definition}
Let $\IRS{G}$ denote the space of all invariant random subgroups of $G$ equipped with the weak-$*$ topology.  By Riesz' representation
theorem and Alaoglu's theorem, $\IRS{G}$ is a compact space.  



\subsection{Weak triviality and weak uniform discreteness}

Let $G$ be any l.c.s.c. group. We study weak-$*$ convergence to $\delta_{\{e\}}$ in the space $\mathrm{IRS}(G)$ and relate this to the notion of weak uniform discreteness.

\begin{definition}
	\label{def: a Farber sequence of invariant random subgroups}
	Given $\mu \in \IRS{G}$ and a subset  $A \subset G$ denote
	$$ p_\mu(A) = \mathrm{Prob}_{\mu}(\{\text{$H \le G$ closed} \: : \: A \cap H  = \emptyset \} ) $$
	A sequence    of invariant random subgroups $\mu_n \in \mathrm{IRS}(G)$ is called
	\begin{itemize}
		\item \emph{weakly trivial}\footnote{Some authors use the terminology \emph{Farber} for a sequence of lattices $\Gamma_n$ so that $\mu_{\Gamma_n}$ is trivial in our sense.}  if $\lim_{n\to\infty} p_{\mu_n}(Q \setminus \{e\}) = 1$ for every compact subset $Q \subset G$,  and
		\item \emph{weakly uniformly discrete} if for every $\varepsilon > 0$ there is an identity neighborhood $U_\varepsilon \subset G$ so that $p_{\mu_n}(U_\varepsilon \setminus \{e\}) > 1 - \varepsilon$ for all $n \in \NN$.
	\end{itemize}
\end{definition}

The notion of weak uniform discreteness was introduced in \cite{gel_kaz} where it is used to provide a simple proof for a generalization of the Kazhdan--Margulis theorem. In fact all \emph{discrete} invariant random subgroups in zero characteristic semisimple analytic groups are weakly uniformly discrete \cite{gel_kaz}.

\begin{prop}
	\label{prop:WUD and compact subsets}
	A weakly trivial sequence of invariant random subgroups weak-$*$ converges to $\delta_{\{e\}} \in \IRS{G}$. The converse direction holds provided that  the sequence is weakly uniformly discrete.
\end{prop}

\begin{proof} 
We claim that $\mu_n \xrightarrow{n\to\infty} \delta_{\{e\}}$ if and only if $\lim_{n\to\infty} p_{\mu_n}(Q) = 1$ for every compact subset $Q \subset G$ such that $e \notin Q$. By definition $\mu_n \xrightarrow{n\to\infty} \delta_{\{e\}}$ if and only if 
$$\liminf_{n \to \infty}\mu_n(\Omega) \ge \delta_{\{e\}}(\Omega) = 1_\Omega(\{e\})$$
for every Chabauty-open subset $\Omega \subset \Sub{G}$. It clearly suffices to consider only such $\Omega$ containing the point $\{e\}$. Observe that $\{e\} \in \mathcal{O}_1(Q) \subset \Omega$ for some compact $Q \subset G$ with $e \notin Q$. Since  $\mu_n(\mathcal{O}_1(Q)) = p_{\mu_n}(Q)$ the claim follows. 
In particular weak-$*$ convergence to $\delta_{\{e\}}$ implies weak triviality.

%
	
For the converse direction assume that $\mu_n$ is weakly uniformly discrete and weak-$*$ converges to $\delta_{\{e\}}$. Let $Q \subset G$ be any compact subset. If $e \not\in Q$ then  $\lim_{n\to\infty} p_{\mu_n}(Q) = 1$ by the above claim. Finally assume that $e \in Q$ and fix $\varepsilon > 0$. By weak uniform discreteness there is an identity neighborhood $U_\varepsilon$ so that $p_{\mu_n}(U_\varepsilon \setminus \{e\}) > 1-\varepsilon$ for all $n \in \NN$.
	Denote $Q_\varepsilon = Q \setminus U_\varepsilon$ so that $Q_\varepsilon$ is compact. Therefore 
$$
	p_{\mu_n}(Q \setminus \{e\})  \ge p_{\mu_n}(U_\varepsilon \setminus \{e\}) + p_{\mu_n}(Q_\varepsilon) - 1 \ge p_{\mu_n}(Q_\varepsilon) - \varepsilon \xrightarrow{n\to\infty} 1 - \varepsilon
$$
	The proposition follows by taking $\varepsilon > 0$ to be arbitrarily small.
\end{proof}

In analogy with Definition \ref{def: a Farber sequence of invariant random subgroups} we say that a sequence $\mu_n$ of invariant random subgroups is \emph{weakly central} if $\lim_{n\to\infty} p_{\mu_n}(Q \setminus Z(G)) = 1$ for every compact subset $Q \subset G$. Similarly to  Proposition \ref{prop:WUD and compact subsets} it is easy to see that any accumulation point of a weakly central sequence is supported on the center, and that the converse holds for weakly uniformly discrete sequences. We remark that dividing a group by its center results in sending a weakly central sequence to a weakly trivial one.


\begin{cor}
	\label{cor:in zero char convergence to 0 implies Farber}
	Let $G$ be a   semisimple analytic group in zero characteristic. Then a sequence of discrete invariant random subgroups weak-$*$ converging to $\delta_{\{e\}}$ is weakly trivial.
\end{cor}


%

 We leave the straightforward verification of Example \ref{exam:residual-chain} to the reader.

\begin{exam}\label{exam:residual-chain}
	\label{prop:invariant random subgroups of a descending sequence go to delta e}
	Let $\Gamma$ be a lattice in $G$ and let $\Gamma_i \nrm \Gamma$ be a descending sequence of normal subgroups with $\bigcap_{i\in\NN} \Gamma_i = \{e\}$. Then $\Gamma_i$ is a weakly  trivial sequence.
	
	In particular, the sequence of invariant random subgroups $\mu_{\Gamma_i}$ weak-$*$ converges to $\delta_{\{e\}}$ by Proposition \ref{prop:WUD and compact subsets}.
\end{exam}

\subsection{Quotient spaces and the pointed Gromov--Hausdorff topology}

Let $G$ be a l.c.s.c. group. Let $(X,x)$ be a  proper  locally compact pointed  metric space admitting a proper continuous $G$-action by isometries. 

\begin{definition}
\label{def:Q(G,X)}
 The \emph{space of quotients} associated to the $G$-action on $X$ is
$$ \mathcal{Q}(G,X) = \{H\backslash X \: : \: \text{$H$ is a closed subgroup of $G$} \} $$
The quotients $H \backslash X$ are regarded as pointed metric spaces with basepoints $Hx$ and the quotient metrics. $\mathcal{Q}(G,X)$ is equipped with the pointed Gromov--Hausdorff topology.
\end{definition}
We will be working with the pointed Gromov--Hausdorff topology in terms of $(\varepsilon,r)$-relations. See  \cite[I.3.2]{canary1986notes} for more details on this notion.

\begin{prop}
\label{prop:map into Q(G,X) is continuous}
The natural map
$$ \Sub{G} \to \mathcal{Q}(G,X), \quad H \mapsto H \backslash X $$ 
is continuous.
\end{prop}
\begin{proof}
Let $H_n$ be a sequence of closed subgroups in $G$ converging to the closed subgroup $H$ in the Chabauty sense. We  show that the quotient pointed metric spaces $H_n\backslash X$ converge to $H\backslash X$ in the pointed Gromov--Hausdorff topology.

Fix some small $\varepsilon > 0$ and large $r > 0$. It will suffice to exhibit $(\varepsilon,r)$-relations between $H_n \backslash X$ and $H\backslash X$ for all $n$ sufficiently large. 
Let $\overline{B}(r)$ denote the closed ball $\overline{B}_X(x,r)$  in the space $X$. Since $X$ is proper $\overline{B}(r)$ is compact. For every closed subgroup $H$ of $G$ denote
$$ Q_H(r) = H \backslash H\overline{B}(r) \subset H \backslash X \quad$$
so that $Q_H(r)$ is compact as well for every $H$. Observe that
$$ Q_H(r) = \overline{B}_{H\backslash X}(Hx,r) $$
where $\overline{B}_{H\backslash X}(Hx,r)$ is the closed $r$-ball centered at $Hx$ in the quotient space $H \backslash X$.

Define a relation $\sim_n$ between the compact subsets $Q_{H_n}(r)$ and $Q_H(r)$ by
$$ H_nx \sim_n Hx, \quad \forall x \in \overline{B}(r) $$
The relations $\sim_n$ need not in general be one-to-one.

We claim that $\sim_n$ is an $(\varepsilon,r)$-relation between the two pointed metric spaces $H_n\backslash X$ and $H \backslash X$ for all $n$ sufficiently large. Let $d$ and $d_n$ denote the  induced metrics on the quotient spaces $H \backslash X$ and  $H_n \backslash X$, respectively. The only non-trivial fact to be proved is that
\begin{equation}\label{eq:QC} 
\abs{d_n(H_n x_1, H_n x_2) - d(H x_1, H x_2) } < \varepsilon, \quad \forall x_1, x_2 \in \overline{B}(r)
\end{equation}
assuming that $n$ is sufficiently large.

Since $G$ acts properly on $X$ there is a compact subset $C\subset G$ such that $g\cdot \overline{B}(3r)\cap\overline{B}(3r)=\emptyset$ whenever $g\notin C$. Let $C'$ be another compact set containing $C$ in its interior. 
The Chabauty convergence of $H_n$ to $H$ implies that for every open identity neighborhood $U \subset G$ we have that
$$ H_n\cap C \subset (H \cap C)U \quad \text{and} \quad H \cap C \subset (H_n \cap C')U $$
for all $n$ sufficiently large.
Note that 
$$
 d(H x_1, H x_2)=\min_{h\in H\cap C} d(x_1,hx_2)=\min_{h\in H\cap C'} d(x_1,hx_2)
$$
and a similar expression holds true for $d_n(H_n x_1, H_n x_2)$ and every $n \in \NN$. Thus (\ref{eq:QC}) follows from the continuity of the $G$-action and the compactness of $C'$.
\end{proof}

We deduce:

\begin{cor}
\label{cor:Q(G,X) is compact}
The space $\mathcal{Q}(G,X)$ is compact.
\end{cor}

Since $G$ is second countable it admits a left-invariant proper metric. Regard $G$ as a pointed metric space with basepoint at the identity. Therefore  the space of quotients $\mathcal{Q}(G,G)$ is well defined and consists of  the quotients $H\backslash G$ where $H$ is a closed subgroup of $G$. The mapping $\Sub{G} \ni H \mapsto H\backslash G$ is continuous. However, note  that in general it need not be injective.

A generalization of the previous construction is obtained as follows. Assume that the metric on $G$ is right-$K$-invariant for some closed subgroup $K$ in $G$. For instance, this is always possible when $K$ is compact. The space of quotients $\mathcal{Q}(G,G/K)$ is well defined and consists of the double coset spaces $H \backslash G / K$ where $H \in \Sub{G}$.

\begin{remark}
The \emph{space of graphs} $\mathcal{GR}(G,X)$ consists of the graphs of the quotient maps $X \to H \backslash X$ regarded as subspaces of $X \times H\backslash X$. These are pointed metric spaces with basepoints $(x,Hx)$. The space $\mathcal{GR}(G,X)$ is equipped with the pointed Gromov--Hausdorff topology. There is a pair of maps
$$ \Sub{G} \to \mathcal{GR}(G,X) \to \mathcal{Q}(G,X) $$
whose composition is the map considered in Proposition \ref{prop:map into Q(G,X) is continuous}. A small modification of the proof of Proposition \ref{prop:map into Q(G,X) is continuous} shows that the two intermediate maps are continuous as well.

The point of this construction is that the map $\Sub{G} \to \mathcal{GR}(G,X)$ is obviously bijective and is therefore a homeomorphism. In particular the topology on $\mathcal{GR}(G,X)$ does not depend on the particular choice of the proper metric on $X$. 
In a vague sense, the topology of $\mathcal{Q}(G,X)$ is also independent of the metric, but the fibers of the map from $\Sub{G}$ may depend on the choice of metric. 

We will not be making use of the space of graphs in this work.
\end{remark}


%
%
%
%
%

\subsection{Benjamini--Schramm convergence}
Let the group $G$ and the pointed metric space $X$ be as above.
The \emph{Benjamini--Schramm space} $\mathcal{BS}(G,X)$ is the space of Borel probability measures on $\mathcal{Q}(G,X)$ with the weak-$*$ topology. Since the space of quotients is compact so is the Benjamini--Schramm space.
Moreover, in view of Proposition \ref{prop:map into Q(G,X) is continuous} we have:

\begin{cor}
\label{cor:map from IRS to BS is continuous}
The natural map $\IRS{G} \to \mathcal{BS}(G,X)$ is continuous.
\end{cor}

The natural map of Corollary \ref{cor:map from IRS to BS is continuous} takes a sequence in $\IRS{G}$ converging to $\delta_{\{e\}}$ to a sequence in $\mathcal{BS}(G,X)$ converging to $\delta_X$. 

The following proposition provides a geometric interpretation for this convergence in the special case of  semisimple analytic groups.

\begin{prop}
\label{prop:comparison of convergence to delta_e and isometric balls}
Let $G$ be a semisimple analytic group and  $X$  the corresponding product of a symmetric space and a Bruhat--Tits building. Every weak-$*$ accumulation point of the sequence   $\mu_n \in \IRS{G}$ is central if and only if  
	$$\mathrm{Prob}_{\mu_n} \left( \text{$B_{H \backslash X}(Hx ,r )$ is isometric to $  B_{X}(x,r)$} \right) \xrightarrow{n\to\infty} 1$$ 	
 for every radius $0 < r < \infty $. Here $x \in X$ is an arbitrary  basepoint.
\end{prop}
\begin{proof}

Write $G = C \times D$ where $C$ and $D$ are the Archimedean and non-Archimedean factors of $G$, respectively. 
For every fixed radius $r > 0$ denote $B_r = \overline{B}_X(x,r)$ and consider the subset
$$  
Q_r = \{ g \in G \: : \: gB_r \cap B_r \neq \emptyset \} \subset G.
$$
The properness of the action implies that $Q_r$ is compact. Consider the subset
$$ 
V_r = \{ g \in D \: : \: g_{| B_r} = \mathrm{id}_{|B_r} \}. 
$$
Note that $V_r$ is a compact open subgroup in $D$. Let $U_r \subset C $ be a neighborhood of $Z(C)$ so that the only closed subgroups of $C$ contained in $U_r$ are central and such that $U_r \times V_r \subset Q_r$. Assume that $U_{r'} \subset U_r$ for every $r' > r$ and $\bigcap_r U_r = Z(C)$.

Consider the Chabauty open subset $\Omega_r = \mathcal{O}_1(Q_r \setminus (U_r \times V_r))$. A closed subgroup $H$ belongs to $\Omega_r$ if and only if $H \cap Q_r \subset V_r$ if and only if the $r$-balls at the basepoints of $H \backslash X$ and of $X$ are isometric.
Just as in the proof of Proposition \ref{prop:WUD and compact subsets} the fact that every accumulation point of $\mu_n$ is central is equivalent to $\lim_{n\to\infty} p_{\mu_n}(Q) = 1$ for every compact subset $Q \subset G$ with $Z(G) \cap Q =\{e\}$. Every such compact subset $Q$ is contained in $Q_r\setminus(U_r \times V_r)$ for all $r$ sufficiently large, and the result follows.
\end{proof}

In the situation of Proposition \ref{prop:comparison of convergence to delta_e and isometric balls}  the probability with respect to $\mu_n$ that a random $r$-ball in $H \backslash X$ is contractible tends to one for every radius $0 < r < \infty $.

\section{On Zariski closure and the Chabauty topology} 

\label{sec:zariski closure and lie algebras}

Let $k$ be any local field. We analyze the Zariski closure operation on the Chabauty space of closed subsets in the group $\GG(k)$ of $k$-rational points  of an algebraic $k$-group $\GG$. As an application we define a certain variant ---  {\it local} Zariski closure operation --- which turns out to be useful in positive characteristic.

The section begins with a rather general discussion on $k$-vector spaces. As we progress we specialize to study the linear representation of the group $\GG(k)$ in the $k$-linear space of its regular functions arising from the action of $\GG$ on itself by conjugation.
Some of our arguments are inspired by \cite[\S 4]{burger}.



%

\subsection{Limits of finite dimensional vector spaces}

Consider a sequence $A_i$ of finite dimensional $k$-vector spaces for $i \in \NN$. Assume that this sequence admits a pair of compatible structures as a direct and an inverse system. That is, there are $k$-linear maps $\iota_i : A_i \to A_{i+1}$ and $\sigma_i : A_{i + 1} \to A_i $ that satisfy $\sigma_i \circ \iota_i = \mathrm{id}_{|A_i}$ for all $i \in \NN$. In particular $ \iota_i \circ \sigma_i $ is an idempotent $k$-endomorphism of $A_{i+1}$ for every $i \in \NN$.

The Grassmannian $\Gr{A_i}$  is by definition the disjoint union of its $d$-dimensional parts with $0 \le d \le \dim A_i$ for every $i \in \NN$.   We obtain a compatible direct and inverse system structures on the sequence $\mathcal{GR}(A_i)$ of Grassmannians. We retain the notations $\iota_i$ and $\sigma_i$ for the induced maps in these systems.

Let $A$ denote the $k$-vector space direct limit $A = \varinjlim_i A_i$. The Grassmannian $\mathcal{GR}(A)$ consists of all the $k$-vector subspaces of $A$ and it is easy to verify that $\mathcal{GR}(A) = \varprojlim_i \mathcal{GR}(A_i)$. Since $k$ is a local field every $\mathcal{GR}(A_i)$ is compact. In particular $\mathcal{GR}(A)$ becomes a compact space with the inverse limit topology.


The dual $k$-vector space  of $A$ is the inverse limit $A^* = \varprojlim_i A_i^*$. This is a topological space with the inverse limit topology. As before let $\mathcal{GR}(A^*)$ denote the Grassmannian of all $k$-vector subspaces of $A^*$. There is an order-reversing bijection 
$$\mathrm{ann} : \mathcal{GR}(A^*)  \to \mathcal{GR}(A)$$ 
obtained by taking  annihilators. We use the notation $\iota^*_i$ and $\sigma^*_i$ for the maps in the dual inverse and direct systems, as well as the induced maps in the system $\mathcal{GR}(A_i^*)$.

%
%
%

\subsection{The linear span map}

Consider the linear span map denoted $\mathrm{sp}$ and defined on the power set  of $A^*$ by
	$$ \mathrm{sp} : \mathrm{Pow}(A^*) \to \Gr{A^*}, \quad \mathrm{sp}(F) = \mathrm{span}_k(F) \quad \forall F \subset A^*.$$

Given a topological space  $X$ we let $\Cl{X}$ denote  the space of closed subsets of $X$  regarded with the Chabauty topology, as in Definition \ref{def:Chabauty topology}.

\begin{prop}
	\label{prop:linear closure map is measurable}
Let $X$ be a topological space and	 $\varepsilon: X \to A^*$ a continuous map.    Denote by $\hat{\varepsilon}$ the corresponding set map $\Cl{X} \to \mathrm{Pow}(A^*)$. Then the composition
	$$  \mathrm{ann} \circ \mathrm{sp} \circ \hat{\varepsilon} : \Cl{X} \to \mathrm{Pow}(A^*) \to \Gr{A^*} \to \mathcal{GR}(A)$$ 
	is Borel measurable. 
\end{prop}

Note that we have not explicitly defined a Borel structure on the two intermediate spaces $\mathrm{Pow}(A^*)$ or on $\mathcal{GR}(A^*)$.

\begin{proof}
It will suffice to show that the composition of $\mathrm{ann} \circ \mathrm{sp} \circ \hat{\varepsilon}$ with the projection $\mathcal{GR}(A) \to \mathcal{GR}(A_i)$ is upper semi-continuous, and in particular Borel measurable, for every $i \in \NN$. This composition, denoted $\alpha_i$, is equal to 
$$ \mathrm{Cl}(X) \xrightarrow{\hat{\varepsilon}} \mathrm{Pow}(A^*) \to \mathrm{Pow}(A^*_i) \xrightarrow{\mathrm{sp}_i} \mathcal{GR}(A_i^*) \xrightarrow{\mathrm{ann}_i} \mathcal{GR}(A_i) $$
where the second map is induced from the projection $A^*\to A^*_i$, $\mathrm{sp}_i$ is the $k$-linear span map on $A^*_i$ and $\mathrm{ann}_i$ is the annihilator map between $\mathcal{GR}(A^*_i)$ and $\mathcal{GR}(A_i)$ which is an order-reversing homeomorphism.

We need to show that $\alpha_i$ is  lower semi-continuous. 
Consider a closed subset $F \subset X$ with $\dim \alpha_i(F) = d$ where $0\le d \le \dim A_i^*$. There are $d$ points $x_1, \ldots, x_d \in F$ such that 
$$\alpha_i(F) = \mathrm{sp}_i \circ \hat{\varepsilon}(\{x_1,\ldots,x_d\}).$$
The continuity of the map $\varepsilon$ implies that the subspace $\mathrm{sp}_i \circ \hat{\varepsilon} (\{x'_1,\ldots,x'_d\}) $ is close to $\alpha_i(F)$ in the topology of $\mathcal{GR}(A_i^*)$. So these two subspaces have the same dimension $d$ for points $x'_i$ sufficiently close to $x_i$ in $X$. This determines a Chaubuty neighborhood $\Omega$ of $F$ so that
$ \alpha_i(F') \subset \mathrm{sp}_i \circ \hat{\varepsilon} (\{x'_1,\ldots,x'_d\})$ for $F' \in \Omega$ as required.
\end{proof}

%

\subsection{Regular functions on affine spaces}
\label{sub:algebraic varieties}

We specialize the  discussion to our case of interest, which is the algebra of $k$-regular functions on an affine space.

Consider the algebra $A$ of $k$-regular functions on the $N$-dimensional affine space 
$$A = k\left[x_1,\ldots,x_N\right]$$
for some fixed $N \in \NN$.  For our purposes $A$ is regarded simply as an $k$-vector space, and it is in fact the direct limit of the finite dimensional $k$-vector spaces
$$ A_i = \{p\in  k\left[x_1,\ldots,x_N\right]\: : \: \deg p \le i \}$$
with the direct system of maps $\iota_i : A_i \to A_{i+1}$ being the inclusion. An inverse system of maps $\sigma_i : A_{i+1} \to A_i$ is given by
$$ p = \sigma_i(p) + r,  \quad \text{s.t. $\sigma_i(p) \in A_i$ and $r$ has no non-zero monomials of degree $\le i$}  $$
for all $i \in \NN$ and polynomials $p \in A_{i+1}$. The two systems $\iota_i$ and $\sigma_i$ are compatible in our previous sense, that is $\sigma_i \circ \iota_i = \mathrm{id}_{|A_i}$.

Let $\mathrm{I}(F)$ denote the ideal in $A$ of vanishing $k$-regular functions on the subset $F$  of the affine space $k^N$.  In other words
$$ 
 \mathrm{I}(F) = \{ p \in A \: : \: p(x) = 0 \quad \forall x \in F \}.
$$
By definition $\mathrm{I}(F) = \mathrm{I}(\overline{F}^{\mathrm{Z}})$ where $\overline{F}^{\mathrm{Z}}$ is the Zariski closure of $F$ in $k^N$.



\begin{prop}
	\label{prop:Zariski closure map is measurable}
Consider $k^N$ with the topology arising from the local field $k$. Then the map 
	$ \mathrm{I} : \Cl{k^N} \to \Gr{A} $ is Borel measurable.
\end{prop}
\begin{proof}
The evaluation map $\varepsilon : k^N \to A^*$ is defined by
	$$ 
	  \varepsilon(x)(p) = p(x) \quad \forall x \in k^N, p \in A.
	$$
Extend $\varepsilon$ to a set map $\hat{\varepsilon} : \Cl{k^N} \to \mathrm{Pow}(A^*)$. Observe the identity
$$ 
 \mathrm{I} = \mathrm{ann} \circ \mathrm{sp} \circ \hat{\varepsilon} : X \to \Gr{A}.
$$
The fact that $I$ is Borel measureable follows from Proposition \ref{prop:linear closure map is measurable}.
\end{proof}

The proof of Proposition \ref{prop:linear closure map is measurable}  gives more information, namely that the composition of $\mathrm{I}$ with a projection $\mathcal{GR}(A) \to \mathcal{GR}(A_i)$ is upper semi-continuous.

\subsection{Actions with fixed points on affine spaces}
\label{sub:algebraic actions}

Let $\mathrm{GL}(A)$ denote the group of all $k$-linear automorphisms of $A$. $\mathrm{GL}(A)$ admits a subgroup $\mathrm{GL}^*(A)$ given by
$$ \mathrm{GL}^*(A) = \{ \alpha \in \mathrm{GL}(A) \: : \: \sigma_i \alpha = \sigma_i \alpha \sigma_i  \quad \forall i \in \NN \} $$
where  $\sigma_i$ is the map $ A \to A_i$ and $A_i$ is considered as a subspace of $A$. The subgroups $\mathrm{GL}^*(A_i)$ of $\mathrm{GL}(A_i)$ may be analogously defined for every $i \in \NN$. For $\alpha \in \mathrm{GL}^*(A)$ we set $\alpha_i = \sigma_i \alpha \in \mathrm{GL}^*(A_i)$ for every $i \in \NN$. Observe that $\mathrm{GL}^*(A) = \varprojlim_{i} \mathrm{GL}^*(A_i)$.

Let $G$ be any group admitting an  action on the $N$-dimensional affine space defined over $k$.
We obtain a representation $\tau$ of $G$ in the $k$-vector space $A$. Namely $\tau$ is a group homomorphism 
$$ \tau :  G \to \mathrm{GL}(A), \quad \tau(g)(p) = p \circ f_{g^{-1}}, \quad \forall p \in A $$
so that for every element $g \in G$ the map $f_g = (f_g^1,\ldots,f^N_g)$ is an automorphism from $k^N$ to itself given by $k$-polynomials.

\begin{prop}
\label{prop:homogenous action resticts to subspaces}
If the action of $G$  fixes the point $0 \in k^N$ then
$ \tau(G)  \subset \mathrm{GL}^*(A)$.
\end{prop}

\begin{proof}
We need to verify that 
$\sigma_i \tau(g) \sigma_i = \sigma_i \tau(g)$ holds for every element $g \in G$ and every $i \in \NN$.
The fixed point assumption means that the polynomials $f_g^j$  have a vanishing constant term for every element $g \in G$ and every index $1 \le j \le N$, and this implies the required condition.
\end{proof}

The projectivization of the representation $\tau$ gives rise to a natural action $\overline{\tau}$ of $G$ on the Grassmannian $\Gr{A}$.

%


\subsection{The conjugation action of an algebraic group}
\label{sub:action by conjugation}

We further specialize our discussion to the conjugation action of an algebraic group on itself.

Let $\GG$ be a linear group defined over $k$. We may assume that $\GG$ is affine \cite[1.10]{borel} and that $\GG \le \mathrm{GL}_M$ for some $M \in \NN$. Denote $G = \GG(k)$ so that 
$$G \subset \mathrm{M}_{M}(k) \cong k^N \quad \text{where $N = M^2$} .$$

Consider the action of $G$ on $\mathrm{M}_{M}(k) \cong k^N$ by matrix conjugation. This action is  defined over $k$ and fixes zero. There is a corresponding  representation $\tau$ of the group $G$ in the the $k$-vector space $A = k\left[x_1,\ldots, x_N \right]$.
According to Proposition \ref{prop:homogenous action resticts to subspaces}  in fact $\tau(G) \subset  \mathrm{GL}^*(A)$ and $\tau = \varprojlim_{i} \tau_i$ for a family of representations $\tau_i : G \to \mathrm{GL}^*(A_i)$. The representations $\tau_i$ are continuous.



\begin{definition}
\label{def:map phi}
The map $\mathrm{I}_G : \Sub{G} \to \Gr{A}$ is the composition
$$ \mathrm{I}_G : \Sub{G}  \to \Cl{k^N}   \xrightarrow{\mathrm{I}} \Gr{A} $$
where the map $\mathrm{I}$ is given prior to  Proposition \ref{prop:Zariski closure map is measurable} above.
\end{definition}

Clearly $\mathrm{I}_G(G) \subset \mathrm{I}_G(F)$ for every closed subset $ F $ in $G$. Moreover, for every closed subgroup $H$ in $G$ the ideal $\mathrm{I}_G(H) \nrm A$ corresponds to its Zariski closure $\overline{H}^{\mathrm{Z}}$. Of course, $\overline{H}^{\mathrm{Z}}$ is a Zariski closed subgroup of $G$.
The map $\mathrm{I}_G$ is $G$-equivariant with respect to   the conjugation action on $\Sub{G}$ and the $\bar{\tau}$-action on $\Gr{A}$.

\subsection{Intersections of conjugates}

%
%
%

As an application of the above discussion, we obtain a result on the intersection of finitely many conjugates of a Zarsiki closed subgroup. This result will be later used in the proof of Lemma \ref{lem:on intersections of parabolics}.

\begin{prop}
	\label{prop:intersections over deformed elements}
	Let $N$ and $H$ be a pair of  Zariski closed subgroups in $G$ so that $N \le H$ and $N \nrm G$. 
	Let $\Gamma \le G$ be a Zariski dense subgroup. Assume that
	$$ 
	 \bigcap_{\gamma \in \Gamma}H^{\gamma} = N.
	$$
	Then there is some $m \in \NN$, elements $\gamma_1, \ldots, \gamma_m \in \Gamma$ and an identity neighborhood $W$ in $G$ so that
	$$ \bigcap_{n=1}^m H^{\gamma_n w_n} = N $$
	holds for every choice of elements  $w_n \in W$.
\end{prop}

For the proof recall that $\bar{\tau}$ is the  representation of $G$ on the Grassmaniann space $\Gr{A}$ arising from $\tau$.  In fact $ \bar{\tau} = \varprojlim \bar{\tau}_i$ where $\bar{\tau_i}$ are the projective actions of $G$ on the Grassmanianns $\mathcal{GR}(A_i)$ arising from the finite dimensional $\tau_i$'s.

\begin{proof}[Proof of Proposition \ref{prop:intersections over deformed elements}]
	Denote $I_H = \mathrm{I}_G(H)$. This is a 
	 point in  $\Gr{A}$ corresponding to an ideal in $A = \varinjlim_{i \in \NN}A_i$. Consider the ideal 
	$$ 
	 J =  \sum_{\gamma \in \Gamma} \bar{\tau}(\gamma) I_H \in \Gr{A}.
	$$
	Clearly $J$ is a $\bar{\tau}(\Gamma)$-invariant subspace of $A$. By assumption $J$ determines the Zariski closed subgroup $N$. The Noetherianity of $A$ implies that $J$ admits a finite generating set $S$ and that moreover $ J = \sum_{n=1}^m \bar{\tau}(\gamma_n) I_H$ for some $m \in \NN$ and elements $\gamma_1, \ldots, \gamma_m \in \Gamma$.

	Choose a sufficiently large index $i \in \NN$ so that $S \subset  A_i$ and let $\sigma_i$ denote the natural projection $\Gr{A} \to \Gr{A_i}$. Observe that
	$$ 
	 S = \sigma_i(S) \subset  \sigma_i(J) = \sum_{n=1}^m \bar{\tau}_i(\gamma_n) \sigma_i(I_H).
	$$
	The projective representation $\overline{\tau}_i$ is continuous. Since $\sigma_i(J)$ is  $\bar{\tau}_i(\Gamma)$-invariant the classical Borel density theorem \cite[II.4.4]{Ma} implies that $\sigma_i(J)$ is $\bar{\tau}_i(G)$-invariant as well. Therefore there is a sufficiently small identity neighborhood $W \subset G$ so that the above condition continues to hold up to passing to  elements of the form $g_n w_n $ with $w_n \in W$.	
		This shows that  the ideal $I' = \sum_{n=1}^m \bar{\tau}(\gamma_n w_n) I_H$ contains the generating set $S$ and in particular that $J = I'$. The required conclusion follows.
\end{proof}

\subsection{Local Zariski closure}

A certain modification $\mathrm{J}_G$ of the map  $\mathrm{I}_G$ introduced above allows us to use Zariski closure in a local manner. Fix  a countable local basis $U_i$ of neighborhoods  at the identity for $G$.

\begin{definition}
	\label{def:map Psi}

  The map $\mathrm{J}_G$ is given by
	$$ \mathrm{J}_G : \Sub{G} \to \Gr{A}, \quad \mathrm{J}_G(H) = \mathrm{span}_{i} \, \mathrm{I}(H \cap \overline{U}_i) $$

\end{definition}

Note that this definition  is independent of the particular choice of a local basis.
Our treatment of the map $\mathrm{I}_G$ applies to the map $\mathrm{J}_G$ as well. In particular, it follows from Proposition \ref{prop:linear closure map is measurable} that $\mathrm{J}_G$ is Borel measurable and $G$-equivariant for the corresponding actions.

\begin{prop}
	\label{prop:on map Psi and noetherianity}
Let $G$ be a semisimple analytic group over the local field $k$ and $H$ a closed subgroup in $G$.
\begin{enumerate}
	
\item If $k$ is Archimedean then $\mathrm{J}_G(H) = \mathrm{I}_G(H_0)$, and
\label{itm:arch case}

	\item If $k$ is non-Archimedean then there is a compact open subgroup $K_H \le G$  
	depending on $H$ so that
	$ \mathrm{J}_G(H) = \mathrm{I}_G(H \cap K_H) $.
	\label{item:non arch case}
\end{enumerate} 
\end{prop}
\begin{proof}
(\ref{itm:arch case}) In the Archimedean case the connected component at the identity $H_0$ is open in $H$ and therefore $\mathrm{J}_G(H) = \mathrm{J}_G(H_0)$.
In zero characteristic, a subset of an irreducible variety which is  open in the Hausdorff $k$-topology is Zariski dense \cite[Lemma 3.2]{pr}. In particular $\mathrm{J}_G(H_0) = \mathrm{I}_G(H_0)$.

(\ref{item:non arch case}) Noetherianity implies that  there is an identity neighborhood $U_H$ in $G$ depending on $H$ so that $\mathrm{J}_G(H) = \mathrm{I}(H \cap \overline{U_H})$. In the non-Archimedean case we may assume that $K_H = \overline{U_H}$ is a compact open subgroup.
\end{proof}

Note that in the non-Archimedean case $\mathrm{J}_G(H) = \mathrm{I}_G(H \cap K'_H)$ as well for every compact open subgroup  $K'_H$ contained in $ K_H$.

\subsection{Totally disconnected factors and local Zariski closure}

Given any totally disconnected l.c.s.c.  group $D$  the definition of the map $\mathrm{J}_G$ can be generalized by extending its domain  to the product Chabauty space $\Sub{G\times D}$. This will be useful in the proof of Theorem \ref{thm:borel density for IRS} in \S \ref{sec:borel density theorem for IRS} below.
Denote $G_1 = G \times D$ and fix a countable local basis of identity neighborhoods $V_i$ for $G_1$.
\begin{definition}
	\label{def:map Psi for products}
 The map $\mathrm{J}_{G_1, G}$ is given by
	$$ \mathrm{J}_{G_1,G} : \Sub{G_1} \to \Gr{A}, \quad \mathrm{J}_{G_1,G}(H) = \mathrm{span}_{i} \, \mathrm{I}(\mathrm{pr}_G (H \cap \overline{V}_i) ) $$
 Here $\mathrm{pr}_G$ is the projection map $G_1 = G \times D \to G$.
\end{definition}

Our discussion of the map $\mathrm{J}_G$ naturally extends to the map $\mathrm{J}_{G_1,G}$. In particular this map is measurable and equivariant. If $G$ itself is non-Archimedean then an analogue of Proposition \ref{prop:on map Psi and noetherianity} holds true for $\mathrm{J}_{G_1,G}$ in the sense that  every closed subgroup $H \le G_1 = G \times D$ satisfies $\mathrm{J}_{G_1,G}(H) = \mathrm{I}_G(\mathrm{pr}_G(H \cap K_H))$   with some compact open subgroup $K_H \le G_1$ depending on $H$.


%
%

\section{Borel density theorem for invariant random subgroups}
\label{sec:borel density theorem for IRS}

Fix  a local field $k$ and let $G$ a happy semisimple analytic group over $k$. We prove a  result on invariant random subgroups in $G$ generalizing the classical Borel density theorem. The main difficulty is in dealing with the positive characteristic case --- relying on the methods of \S\ref{sec:zariski closure and lie algebras} and on Pink's theorem.

\subsection{Invariant probability measures on Grassmannians}
Morally, the  key idea is coming from  Furstenberg's proof of Borel's theorem \cite{furst_borel}. 



\begin{prop}[Furstenberg's lemma]
\label{prop:grassmanian has no invariant probs}
Let $V$ be a finite dimensional $k$-vector space. Let  $\rho : G \to \mathrm{GL}(V)$ a $k$-rational representation and $\bar{\rho} : G \to \mathrm{PGL}(V)$ the corresponding projective representation. Then every  $\bar{\rho}(G^+)$-invariant probability measure  on $\mathrm{P}(V)$  is supported on $\bar{\rho}(G^+)$-invariant points.
\end{prop}
\begin{proof}
The analogous statement is established in  \cite[Lemma 3]{furst_borel} assuming that $G$ is a minimally almost periodic topological group, i.e. $G$ has no non-trivial continuous homomorphisms into compact groups. This assumption is used twice throughout the proof, and we indicate alternative arguments in our  case. 
We restrict attention to the subgroup $G^+$.
\begin{enumerate}
\item $\rho(G^+)$ is not relatively compact in $\mathrm{GL}(V)$.
Since $G$ is $k$-isotropic it contains some $k$-split torus $S$. The image $\rho(S)$ is $k$-split as well \cite[8.2]{borel} and is in particular not relatively compact.  Since $G^+$ is cocompact in $G$ the image $\rho(S\cap G^+)$ is already not relatively  compact in $\mathrm{GL(V)}$.
\item $G^+$ has no non-trivial homomorphisms into finite groups. Indeed it is known that the group $G^+$ has no finite index subgroups \cite[I.2.3.2]{Ma}.
\end{enumerate}

We can now proceed verbatim as in \cite[Lemma 3]{furst_borel} and deduce that a $\bar{\rho}(G^+)$-invariant probability measure on $P(V)$ is supported on $\bar{\rho}(G^+)$-invariant points. \end{proof}

Assume that  $G$ admits a  $k$-action  on the $N$-dimensional affine space. This gives a representation
$ \tau : G \to \mathrm{GL}(A) $
where $A = k\left[x_1, \ldots, x_N\right]$. 
Recall from \S\ref{sec:zariski closure and lie algebras}
that $A$ is a $k$-vector space direct limit $ A = \varinjlim_{i \in \NN} A_i$. The finite  dimensional $k$-vector spaces  $A_i$ admit a  direct system  $\iota_i$ as well as a compatible inverse system $\sigma_i$.

Assume moreover that $ 0\in k^N$ is a fixed point for the $G$-action. In that case it follows from Proposition \ref{prop:homogenous action resticts to subspaces} that
$\tau(G) \subset \mathrm{GL}^*(A)$ and $ \tau = \varprojlim_i \tau_i$ with $\tau_i : G \to \mathrm{GL}^*(A_i)$.
Let $\bar{\tau}$ denote the  representation of $G$ on the Grassmaniann space $\Gr{A}$ arising from $\tau$.  We obtain
$ \bar{\tau} = \varprojlim \bar{\tau}_i$ where $\bar{\tau_i}$ is an action of $G$ on the Grassmaniann $\mathcal{GR}(A_i)$ arising from $\tau_i$.
%
%
%
%

\begin{prop}
\label{prop:fixed point of Grassmannian supported on single point}	
Let $\mu$ be a $\bar{\tau}(G^+)$-invariant probability measure on $\Gr{A}$. If $0 \in k^N$ is a $G$-fixed point then $\mu$ is supported on $\bar{\tau}(G^+)$-fixed points.
\end{prop}
\begin{proof}
A standard  ergodic decomposition argument allows us to assume without loss of generality that  $\mu$ is ergodic with respect to $\overline{\tau}(G^+)$.

Pushing forward $\mu$ by means of the inverse system of maps $\Gr{A} \to \Gr{A_i}$ we obtain ergodic $\bar{\tau}_i(G^+)$-invariant probability measures $\mu_i$ on $\Gr{A_i}$  for every $i \in \NN $. Ergodicity  implies that  there are numbers $0 \le d_i \le \mathrm{dim} A_i$ so that $\mu_i$ is supported on the $d_i$-dimensional part of $\Gr{A_i}$.

Let $\rho_i = \wedge^{d_i} \tau_i$ denote the wedge representation of $G^+$ in $V_i = \wedge^{d_i} A_i $ and $\bar{\rho_i}$ the projectivization of $\rho_i$. The Pl\"{u}cker embedding allows us to identify the $d_i$-dimensional part of the Grassmannian $\Gr{A_i}$ with a subset of the projective space $\mathrm{P}(V_i)$. Under this identification $\mu_i$ can be regarded as a $\bar{\rho}_i(G^+)$-invariant probability measure on $\mathrm{P}(V_i)$. By Furstenberg's lemma and since $\mu_i$ is ergodic it must be supported on a single point corresponding to a $d_i$-dimensional subspace $I_i \le A_i$.

Passing to the inverse limit we obtain that $\mu$ is an atomic point mass supported on some $\tau(G^+)$-fixed point $I \in \Gr{A}$.
\end{proof}


\subsection{Proof of the Borel density theorem for IRS}

We present a proof of Theorem \ref{thm:borel density for IRS} relying on the above version of Furstenberg's lemma as well as Pink's theorem. 
We start with some preparation.

\begin{prop}
	\label{prop:an IRS is either normal or finite}
	Let $G^\dagger$ be an intermediate subgroup $G^+ \le G^\dagger \le G$. Let $F$  be  an almost direct factor of $G$. 
		Let $\mu$ be an ergodic invariant random subgroup of $G^\dagger$.
Then there is a pair of Zariski closed normal subgroups $N,M \nrm G$ with $N \le M$ so that  
	\begin{itemize}
		\item in case $k$ is Archimedean $N = \overline{H_0}^{\mathrm{Z}}$, 
		\item in case $k$ is non-Archimedean $N = \overline{\mathrm{pr}_F (H \cap K_H) }^\mathrm{Z}$ for some compact open subgroup $K_H$ depending on $H$, and
		\item $M = \overline{H}^{\mathrm{Z}} $ in both  cases
	\end{itemize}
	for  $\mu$-almost every closed subgroup $H$ in $G^\dagger$.
%
\end{prop}

Here $\mathrm{pr}_F$ denotes the natural projection $G \to F$.

\begin{proof}
We first deal with Zariski closure and the normal subgroup $M$. 

Recall the map $ \mathrm{I}_G : \Sub{G} \to \Gr{A} $ introduced in \S\ref{sec:zariski closure and lie algebras}. Given a closed subgroup $H$ in $G$ the point $\mathrm{I}_G(H) $ corresponds to the ideal of vanishing regular functions associated with the Zariski closure $\overline{H}^{\mathrm{Z}}$  and regarded as a $k$-linear subspace of $A$.


Consider $\mu$ as a probability measure on $\Sub{G}$ and let $\nu$ denote the push-forward  ergodic $\bar{\tau}(G^\dagger)$-invariant  probability measure on $\Gr{A}$ given by
	$ \nu = (\mathrm{I}_G)_* \mu$.

The measure $\nu$ is supported on $\overline{\tau}(G^+)$-invariant points in $\Gr{A}$ according to Proposition \ref{prop:fixed point of Grassmannian supported on single point}. Every  $\overline{\tau}(G^+)$-invariant point $I \in \Gr{A}$ satisfies $\mathrm{I}_G(G) \le I$ and corresponds to a Zariski closed subgroup of $G$ normalized by $G^+$. Such points are in fact $\overline{\tau}(G)$-invariant, by Proposition \ref{prop:normal subgroups of $G$}. The ergodicity of $\nu$ implies that $\overline{H}^{\mathrm{Z}} = M$ for some normal subgroup $M \nrm G$ and for $\mu$-almost every $H$.

The conclusion involving the subgroup $N$ follows in an analogous manner relying on the map $\mathrm{J}_G : \Sub{G} \to \Gr{A}$ and  Proposition \ref{prop:on map Psi and noetherianity}. 

	In the non-Archimedean case and taking into account the factor $F$ we rely on the map $\mathrm{J}_{G,F}$ as in Definition \ref{def:map Psi for products} and the remarks following it. In that case Proposition \ref{prop:fixed point of Grassmannian supported on single point} is applied with respect to the simple analytic group $F$.
\end{proof}

The trace field $\mathcal{TF}_\rho$ of a $k$-representation $\rho$ was introduced in Definition \ref{def:trace field}.

\begin{prop}
\label{prop:an ergodic IRS has fixed trace field}
Let  $G^\dagger$ be an intermidiate subgroup $G^+ \le G^\dagger \le G$ and  $\mu$  an ergodic invariant random subgroup of $G^\dagger$. 
Let $S$ be an adjoint absolutely simple factor of $G$ and $\rho$ a $k$-representation of $S$. 

 If $\mathrm{J}_{G,S}(H) = \mathrm{I}_S(S)$ 
 for $\mu$-almost every closed subgroup $H$ then for every  compact open subgroup $U$  of $G$  the trace field $\mathcal{TF}_\rho(\mathrm{pr}_S(H \cap U))$ is   $\mu$-almost everywhere constant.
\end{prop}

As usual $\mathrm{pr}_S$ denotes the projection $G \to S$. The notation $\mathrm{J}_{G,S}$ is introduced in Definition \ref{def:map Psi for products}. Note that the essential value of the trace field does not depend on the choice of  the compact open subgroup $U$. 

\begin{proof}
Let $\Sub{k_+}$ denote the Chabauty space of the local field $k$ regarded as an additive group. We claim that the following map $\mathcal{T}$
$$\mathcal{T} : \Sub{G} \to \Sub{S} \to \Sub{k_+}, \quad  \Gamma \mapsto \mathrm{pr}_S(\Gamma \cap U) \mapsto \mathcal{TF}_\rho(\mathrm{pr}_S(\Gamma \cap U)) $$ 
is Borel measurable. 
Recall that the Chabauty Borel structure is generated by subsets of the form $O_2(U)$ in the sense of Definition \ref{def:Chabauty topology}. It follows immediately that the  map $\Sub{G} \to \Sub{S}$ as above is Borel measurable. 

It remains to verify that the map taking a closed subgroup $L$ of $S$ to its trace field $\mathcal{TF}_\rho(L)$ is Borel measurable as well. To see this, observe that $\mathcal{TF}_\rho(L)  $ is equal to the closure in $k$ of the values obtained by evaluating countably many polynomials at $\mathrm{tr}\rho(l)$ and $\mathrm{tr}\rho(l)^{-1}$ where $l$ ranges over individual elements of $L$. Using the continuity of the map $S \to k, \; S \ni s \mapsto \mathrm{tr}\rho(s)$ we conclude that $\mathcal{T}$ is Borel measurable.

By assumption $\mathrm{J}_{G,S}(H) = \mathrm{I}_S(S)$ for $\mu$-almost every closed subgroup $H$. It follows that $\mu$-almost always $\mathrm{pr}_S(\Gamma \cap U)$ is Zariski-dense in $S$. Relying on Proposition \ref{prop:trace field of finite index subgroups} we deduce that the map $\mathcal{T}$ is invariant under the conjugation action of $G^\dagger$. Ergodicity  implies that it must be essentially constant, as required.
\end{proof}

\begin{remark}
Let $\rho$ be a $k$-representation  and $\mu$ an ergodic invariant random subgroup of $G$. It follows that the trace field $\mathcal{TF}_\rho(H)$ is $\mu$-almost everywhere constant on $\Sub{G}$. However this elementary variant can be proved more directly. 
\end{remark}

\begin{proof}[Proof of Theorem \ref{thm:borel density for IRS}]
Let $k$ be a local field, $G$ a happy semisimple analytic group over $k$ without non-standard isogenies and $\mu$ an ergodic invariant random subgroup of $G$. 
We first deal with the statement involving the normal subgroup $N$ and discreteness. 
The Archimedean and non-Archimedean cases are treated separately.

If $k$ is Archimedean then by Proposition \ref{prop:an IRS is either normal or finite} there is a normal subgroup $N_1 \nrm G$ so that $\overline{H_0}^{\mathrm{Z}} = N_1$ for $\mu$-almost every subgroup $H$ in $G$. There is an almost direct factor $R$ of $G$ so that $R_0=R^+ \le N_1$ and $N_1R/R$ is central (Proposition \ref{prop:normal subgroups of $G$}). 
Cartan's closed subgroup theorem \cite[LG 5.9]{sere} implies that $H_0$ is a Lie subgroup of $G$. Since $\overline{H_0}^{\mathrm{Z}} = N_1$ we have that $\mathrm{Lie}(H_0) = \mathrm{Lie}(N_1) = \mathrm{Lie}(R_0)$ and $H_0 = R_0$ for $\mu$-almost every $H$. 
We conclude the Archimedean case by taking $N = R_0 \nrm G$.

If  $k$ is non-Archimedean the proof is more involved. We reduce the situation to adjoint absolutely simple groups and rely on Pink's theorem, as follows. 

 Let $\overline{G}$ be the adjoint group of $G$. There is a central $k$-isogeny $G \xrightarrow{\overline{p}} \overline{G}$. Write $\overline{G}$ as a direct product of its $k$-simple factors $\overline{G} = \prod_{i \in I} \overline{G}_i$ over the finite set $I$. 
 Denote  $\overline{\mu} = \overline{p}_* \mu$ so that $\overline{\mu}$ is an ergodic invariant random subgroup of $ \overline{p}(G)$.
 
It is well-known \cite[3.1.2]{tits1966classification} that there are finite local field extensions $k_i / k$ and adjoint absolutely simple $k_i$-groups $S_i$ such that $\overline{G}_i$ is $k$-isomorphic to $\mathcal{R}_{k_i/k}(S_i)$ for every $i\in I$. Denote $S =  \prod_{i \in I} S_i$. In particular $\overline{G}$ is isomoprhic to $  \prod_{i \in I} S_i(k_i)$ as a topological group\footnote{Making our customary abuse of notation, $S$ denotes an algebraic group as well as its group of rational points.}. Let $ \mathrm{pr}_i$ denote the projection $\overline{G} \to S_i(k_i)$.


We apply  Proposition \ref{prop:an IRS is either normal or finite}  with respect to the simple factor $S_i$ and the invariant random subgroup $\overline{\mu}$ for every $i \in I$. It implies that  $ \mathrm{pr}_i(H \cap K_H)$ is either  Zariski-dense or trivial in $S_i$ for $\overline{\mu}$-almost every closed subgroup $H$ and  compact open subgroup $K_H$ in $\overline{G}$ depending on $H$. 
The index set $I$ is a disjoint union $I = I_\text{d} \cup I_\text{t}$ with $i \in I_\text{d}$ or $ i \in I_\text{t}$ if $ \mathrm{pr}_i(H \cap K_H)$ is  essentially always Zariski-dense or trivial, respectively. In other words $i \in I_\text{t}$ if and only if $\mathrm{pr}_i(H \cap K_H)$ is trivial for $\overline{\mu}$-almost every closed subgroup $H$ and some suitable compact open subgroup $K_H$.


Let $\rho_i$ be the $k$-representation of $S_i$  provided by Theorem \ref{thm:pink}. The trace field $\mathcal{TF}_{\rho_i}(\mathrm{pr}_i(H \cap K_H))$ is $\overline{\mu}$-almost surely constant for every $ i \in I_\text{d}$ by Proposition \ref{prop:an ergodic IRS has fixed trace field}. Let this trace field be denoted $l_i$ so that $l _i  \subset k_i$ is a local field for $i \in I_\text{d}$. The first part of Theorem \ref{thm:pink} implies that the field extension  $k_i / l_i$ is  finite for $i \in I_\text{d}$. For notational convenience denote $l_i = k_i$ for every $ i \in I_\text{t}$.

Let $\widetilde{S_i}$ denote the universal covering group of $S_i$. There is a central $k_i$-isogeny $\widetilde{s}_i : \widetilde{S}_i \to S_i$ and a generalized commutator map $\widetilde{\left[\cdot,\cdot\right]} : S_i \times S_i \to \widetilde{S}_i$ for every $i \in I$. Denoting $\widetilde{S} = \prod_{i\in I} \widetilde{S}_i$ we obtain the maps $\widetilde{s} : \widetilde{S} \to S$ and $\widetilde{\left[\cdot,\cdot\right]} : S \times S \to \widetilde{S}$. The map $\widetilde{s}$ is continuous and proper on the rational points. 
The closure of the generalized commutator gives rise to a well-defined map $\mathrm{gc} : \Sub{S} \to \mathrm{Sub}(\widetilde{S})$. Observe that the map $\mathrm{gc}$ is equivariant in the sense that 
$$
 g^{-1} \cdot \mathrm{gc}(H) \cdot g = \mathrm{gc}\left(\widetilde{s}(g)^{-1} \cdot H \cdot \widetilde{s}(g)\right), 
$$ for all $g \in \widetilde{S}$ and $H\in\Sub{S}$. 
We take $\widetilde{\mu} = \text{gc}_* \overline{\mu}$ so that $\widetilde{\mu}$ is an ergodic invariant random subgroup of $\widetilde{S}$.

Consider the semisimple analytic group $\widetilde{L}$ obtained by taking restriction of scalars
$$  
 \widetilde{L} = \prod_{i \in I } \mathcal{R}_{k_i / l_i} (\widetilde{S}_i). 
$$
Denote $\widetilde{L}_i = \mathcal{R}_{k_i / l_i} (\widetilde{S}_i)(l_i)$ so that $\widetilde{S}_i$ is isomorphic to $ \widetilde{L}_i$ as a topological group for every $i \in I$.
Let $\widetilde{\mathrm{pr}}_i$ denote the projection from $\widetilde{S}$ to the factor $\widetilde{L}_i$. We argue once more relying on Proposition \ref{prop:an IRS is either normal or finite}, this time with respect to the factor $\widetilde{L}_i$ for $i \in I$ and  the invariant random subgroup $\widetilde{\mu}$, deducing that $\widetilde{\mathrm{pr}}_i(F \cap K_F)$ is either Zariski-dense or central   in $\widetilde{L}_i$ for $\widetilde{\mu}$-almost every 
closed subgroup $F$ in $\widetilde{S}$ and a compact open subgroup $K_F$ depending on $F$. This  depends on whether $\widetilde{\mathrm{pr}}_i(F \cap K_F)$ is infinite or finite, respectively. Since the maps $\widetilde{s}_i$ are proper, the groups $\widetilde{\mathrm{pr}}_i(F \cap K_L)$ are essentially always Zariski-dense or central depending on whether  $ i \in I_\text{d}$ or $i \in I_\text{t}$, respectively.

For every $i \in I_\text{d}$ the groups $\widetilde{\mathrm{pr}}_i(F \cap K_F)$  are Zariski dense in $\mathcal{R}_{k_i/l_i}(\widetilde{S}_i)$ $\widetilde{\mu}$-almost always. The second part of Theorem \ref{thm:pink} implies that $\widetilde{\mathrm{pr}}_i(F \cap K_F)$ is open in $\widetilde{S}_i$.
On the other hand for every  $i \in I_\text{t}$, the groups $\widetilde{\mathrm{pr}}_i(F \cap K_F)$ are essentially always central. Since $K_F$ is allowed to be arbitrary small they are in fact trivial.

To draw conclusions denote $\widetilde{S}_\text{d} = \prod_{i \in I_\text{d}} \widetilde{S}_i$ and let $\widetilde{\mu}_\text{d}$ be the ergodic invariant random subgroup of $\widetilde{S}_\text{d}$ obtained from $\widetilde{\mu}$ by taking intersections with $\widetilde{S}_\text{d}$ regarded as a closed subgroup of $\widetilde{S}$. The above discussion shows that $\widetilde{\mu}_\text{d}$-almost every closed subgroup $F \le \widetilde{S}_\text{d}$ admits an open projection to $\widetilde{S}_i$ for any $i \in I_\text{d}$. 
 However,	 a totally disconnected l.c.s.c. group has only countably many open subgroups. An ergodic probability measure on a countable set must be a point mass supported on a fixed point. Taking into account Proposition \ref{prop:normal subgroups of $G$} we see that $\widetilde{\mu}$-almost every closed subgroup of $\widetilde{S}$ contains $\widetilde{S}_\text{d}$. 

We now translate these conclusions back to our groups of interest $\overline{G}$ and $G$. In particular $\overline{\mu}$-almost every closed subgroup $H$ of $\overline{p}(G)$ contains $S_\text{d}^+$ and $H S_\text{d} / S_\text{d}$ is discrete, where $S_d = \prod_{i \in I_\text{d}} S_i$. Therefore $\mu$-almost every closed subgroup of $G$ contains $N =  \overline{p}^{-1}(S_\text{d}^+) $. Since $G$ is happy the quotient $G/G^+$ is finite. This implies that $H N/N$ is discrete in $G/N$ for $\mu$-almost every subgroup $H$, as required.

Finally, it remains to deal with Zariski closure. The existence of a normal subgroup $M \nrm G$ so that $\overline{H}^{\mathrm{Z}} = M$ for $\mu$-almost every closed subgroup $H$ follows immediately from the corresponding part of Proposition \ref{prop:an IRS is either normal or finite}.
\end{proof}


Whenever Lie algebra methods  are available there is a simpler proof of the discreteness part avoiding Pink's theorem. This is the case for $k = \RR$ and $k = \QQ_p$, and this is the approach taken in \cite{7S} for real Lie groups.

\section{Semisimple analytic groups are self-Chabauty-isolated }
\label{sec:isolated groups}

We study of the following property for semisimple analytic groups.

\begin{definition}
\label{def:isolated}
A l.c.s.c. group $G$ is \emph{self-Chabauty-isolated} if the point $G$ is  isolated in $\Sub{G}$ with the Chabauty topology.
\end{definition}

Note that $G$ is self-Chabauty-isolated if and only if there is a finite collection  of open subsets $U_1,\ldots,U_n \subset G$ so that the only closed subgroup intersecting  every $U_i$ non-trivially is $G$ itself. The main theorem in this section is the following.




\begin{theorem}
\label{thm:G is isolated}
Let $G$ be a happy semisimple analytic group. Then $G^+$ is self-Chabauty-isolated.
\end{theorem}

Recall that in particular if $G$ is simply connected or has zero characteristic then it is happy. We will first consider the non-Archimedean case and then extend this to general semisimple analytic groups, allowing for connected factors.

\subsection{The non-Archimedean case}

Let $G$ be an happy non-Archimedean semisimple analytic group. The Chabauty--isolation of the subgroup $G^+$ is easily deduced from the following three propositions.


%


\begin{prop}
\label{prop:isolated and finitely many supergroups}
Let $H$ be a l.c.s.c. group and $O \le H$ an open subgroup. Assume that there are only finitely many intermediate subgroups $ O \le Q \le H$. If $O$ is self-Chabauty-isolated then $H$ is self-Chabauty-isolated as well.
\end{prop}
The above assumption holds in particular when $O$ is maximal or of finite index.
\begin{proof}
Let $Q_1, \ldots, Q_n$ be   the distinct subgroups with $O \le Q_i \lneq H$. Choose arbitrary elements $h_i \in H \setminus Q_i$ for every $i \in \{1,\ldots,n\}$. 
Since $O$ is self-Chabauty-isolated, there are  open subsets $U_1, \ldots, U_m \subset O$ so that the only closed subgroup  of $O$ intersecting each $U_j$ non-trivially is $O$ itself. Then clearly
$$ \{H\} = \bigcap_{i=1}^n \mathcal{O}_2(U_i) \cap \bigcap_{i=1}^m \mathcal{O}_2(h_i O) $$
and the latter subset is open  in $\Sub{H}$, as required.
\end{proof}

\begin{prop}
\label{prop:a finitely generated prop group is isolated}
Let $P_1,\ldots,P_n$ be finitely generated pro-$p_i$ groups for some prime numbers $p_i$. Then the product $\prod_{i=1}^n P_i$  is a self-Chabauty-isolated group.
\end{prop}
\begin{proof}
Follows immediately from \cite[5.6]{gartside2010counting}.
\end{proof}

\begin{prop}
The group $G^+$ has a self-Chabauty-isolated compact open  subgroup.
\label{prop:G has finitely genereted compact open}
\end{prop}
If  $G$ is  $p$-adic then Proposition \ref{prop:G has finitely genereted compact open} follows from Lazard's theorem   \cite[Theorem 8.1]{prop}. Note that every non-Archimedean semisimple analytic group in zero characteristic  can be regarded as a group defined over $\Qp$ using  restriction of scalars. 

\begin{proof}
Let $S$ be a factor of $G$ so that $S$ is a simple analytic group over some local field $k$. If $k$ has zero characteristic let $p$ be such that $k$ is a finite extension of $\Qp$. Otherwise let $p = \mathrm{char}(k)$.  Taking into account Proposition \ref{prop:a finitely generated prop group is isolated} it will suffice to show that the group $S^+$ admits a finitely generated pro-$p$ group.

Recall that $\mathcal{O}$ is the valuation ring  and $\pi$  a uniformizer element of $k$. Assume that $S = \mathbb{S}(k)$ and $\mathbb{S} \le \mathrm{GL}_N$ for some $N \in \NN$. We claim that the first congruence subgroup $\mathbb{S}(\pi \OO)$ of $\mathbb{S}(\OO)$ is a pro-$p$ group\footnote{See  \cite[Lemma 3.8]{pr} for a different proof of the fact that $\mathbb{S}(\pi \OO)$ is a pro-$p$ in the zero characteristic case.}. 

Consider the descending sequence  of congruence subgroups  $\mathbb{S}(\pi^n \OO) $ for $n \in \NN$. These form a basis of neighborhoods at the identity for $\mathbb{S}(\pi \OO)$. We need to verify that every successive quotient $P_n$ with $n \in \NN$
$$P_n = \mathbb{S}(\pi^n \OO) / \mathbb{S}(\pi^{n+1}\OO)$$
is a $p$-group  \cite[2.3.2]{ribes2000profinite}.  
Let $s_1,s_2 \in \mathbb{S}(\pi^n \OO)$ be a pair of elements. We may write
$ s_i = \mathrm{Id} + \pi^{n} A_i$ for some matrices $A_i \in \mathrm{M}_N(\OO)$ and $i=1,2$. The product $s_1 s_2$ is given by 
$$ s_1 s_2 = \mathrm{Id} + \pi^n(A_1 + A_2) + \pi^{n+1}\mathrm{M}_N(\OO) $$
so that $s_1 s_2 = \mathrm{Id} + \pi^n(A_1 + A_2)$ in $P_n$. This calculation shows that $P_n$ is an abelian group of exponent $p$ for every $n \in \NN$, establishing the above claim.

Since $G$ is happy the  Barnea--Larsen theorem, given here as Theorem \ref{thm:properties of happy groups}, implies that the first congruence subgroup is finitely generated. 
\end{proof}

We are ready to show that $G^+$  is self-Chabauty-isolated. The proof should be compared with \cite[3.3]{gel_minsky}.

\begin{theorem}
\label{thm:seclusion of non-Archimedean semisimple}
Let $G$ be a happy non-Archimedean semisimple analytic group. Then the subgroup $G^+$  is self-Chabauty-isolated.
\end{theorem}
\begin{proof}
Let $P $ be a self-Chabauty-isolated compact open subgroup of $G^+$  provided by Proposition \ref{prop:G has finitely genereted compact open}. Let $K$ be a maximal compact open subgroup so that $P \le K \le G^+$. Clearly $P$ has finite index in $K$. According to Corollary \ref{cor:finitely many supergroups} there are only finitely many intermediate subgroups between $K$ and $G^+$. The statement now follows by applying Proposition \ref{prop:a finitely generated prop group is isolated} twice with respect to the two pairs $P \le K$ and $K \le G^+$.
\end{proof}

\subsection{General semisimple analytic groups} We complete the proof of Chabauty isolation for general happy semisimple analytic group.

\begin{proof}[Proof of Theorem \ref{thm:G is isolated}]
Write $G^+$ as an almost direct product $G^+ = C \times D$ where $C$ and $D$  are  connected and totally disconnected semisimple analytic groups, respectively.
The two factors $C$ and $D$ are  self-Chabauty-isolated by \cite[Prop. 2.2]{7S} and  Theorem \ref{thm:seclusion of non-Archimedean semisimple}, respectively. We are required to show that the almost direct product $G^+ = C \times D$ is self-Chabauty-isolated as well.

Let $U_1, \ldots, U_n \subset C$ be open subsets so that the only closed subgroup of $C$ intersecting non-trivially each $U_i$ is $C$ itself. Similarly let $V_1, \ldots, V_m \subset D$ be open subsets satisfying the analogous property with respect to $D$.

By Proposition \ref{prop:G has finitely genereted compact open} there is a self-Chabauty-isolated compact open subgroup  $K \le D$. This implies that $K$ admits an open normal subgroup $P$ so that the only closed subgroup of $K$ intersecting non-trivially every coset of $P$ is $K$ itself. Let $k_1, \ldots, k_t$ be coset representatives for $P$ in $K$.

Consider the Chabauty-open subset $\Omega$ of $\Sub{G^+}$ given by
$$ 
 \Omega = \bigcap_{i=1}^n \mathcal{O}_2(U_i \times P) \cap \bigcap_{i=1}^m \mathcal{O}_2(C \times V_j) \cap   \bigcap_{i=1}^t \mathcal{O}_2(C \times k_i P).
$$
We complete the proof by showing that $\Omega = \{G^+\}$. 

Let $L \in \Omega$ be any closed subgroup. Choose arbitrary elements $ l_{i} \in L \cap (U_i \times P)$ for all $ 1 \le i \le n$ and denote $L' = \left<l_i\right>$. Observe that $L' \le C \times P$ and that $L'$ projects densely to $C$. Next, choose elements $m_i \in L \cap (C \times k_iP)$ for $ 1 \le i \le t$. Note that
$$ 
 \text{$e_{C} \in \overline{\mathrm{pr}_{C}(m_{i}L')}$ and $\mathrm{pr}_{D}(m_{i}L') \subset k_i P$}.
$$
Since $L$ is closed and  the cosets $k_iP$ are compact we deduce that
$$ L \cap (\{e_{C}\} \times k_i P)  \neq \emptyset $$
for every $1 \le i \le t$.  Therefore $L$ contains the open subgroup $\{e_{C}\} \times K$.

The two projections of $L$ to $C$ and $D$ are dense by the definition of $\Omega$. In fact  $L \cap (C \times P)$ already projects densely to $C$. The compactness of $P$ implies that $\mathrm{pr}_{C}(L) = C$. Since $L \cap (\{e_{C}\} \times D)$ is relatively open in $\{e_{C}\} \times D$ we have that $\mathrm{pr}_{D}(L) = D$.

We have established that $L$ surjects onto both factors $C$ and $D$. This implies that
$ L \cap (\{e_{C}\} \times D)$ is normal regarded as a subgroup of  the factor $D$. The normal closure of $K$ in $D$ is the group $D$ itself. So $L$ contains $\{e_{C}\} \times D$. Since $L$ surjects onto the factor $C$ it follows that $L = C \times D = G^+$, as required.
\end{proof}

Since $G$ is happy the quotient $G/G^+$ is finite by Theorem \ref{thm:properties of happy groups}. So that every subgroup $G^\dagger$ with $G^+ \le G^\dagger \le G$ is clearly self-Chabauty-isolated as well.

\begin{cor}[Chabauty--isolation from discrete subgroups]
	\label{cor:isolation of discrete from normal}
	Let $G$  be a happy semisimple analytic group, $N$ a non-discrete normal subgroup of $G$ and $H$ a closed subgroup of $G$ containing $N$ so that $H/N$ is discrete in $G/N$. Then $H$ cannot be approximated by discrete subgroups in $\Sub{G}$.
\end{cor}
\begin{proof}

As was already observed above, the case where $G^+ \le N $ follows immediately from Theorem \ref{thm:G is isolated} combined with Proposition \ref{prop:isolated and finitely many supergroups} and the finiteness of $G/G^+$.

In all other cases  we may according to Proposition \ref{prop:normal subgroups of $G$} decompose $G$ into an almost direct product $G = G_1 \times G_2$ of semisimple analytic groups such that $G_1^+ \le N$ and $H/G_1$ is discrete in $G / G_1$. The proof is by a few reduction steps.

If $G_2$ has a non-Archimedean semisimple analytic factor $G_{2,\text{td}}$ we may write $G$ as an almost direct product $G = G_{0} \times G_{2,\text{td}}$. Let $K$ be any compact open subgroup of $G_{2,\text{td}}$. Note that the map 
$$\Sub{G} \to \Sub{G_{0}}, \quad \Gamma  \mapsto \mathrm{pr}_{G_{0}}(\Gamma \cap  (G_{0} \times K))  $$ is continuous and maps discrete subgroups in $G$ and to discrete subgroups in $G_{0}$, and likewise discrete subgroups in $G/N$ to discrete subgroups $G_{0}/N$. Therefore we may assume in what follows below that $G_2$ is Archimedean.

If $G_1$ is Archimedean as well then $G$ is a semisimple real Lie group. In this case it is a well known fact that the Chabauty limit of discrete groups has a nilpotent connected component and therefore cannot contain $(G_{1})_0$, see e.g. \cite[2.3]{7S}.

If $G_1$ has a non-trivial Archimedean and a non-trivial non-Archimedean factors then $G_1$ splits into an almost direct product $G_1 = G_{1,\text{c}} \times G_{1,\text{td}}$. We may combine the arguments of the two previous cases --- namely, by intersecting with $G_{1,\text{c}} \times K$ for a compact open subgroup $K \le G_{1,\text{td}}$, passing to the quotient $G_1/G_{1,\text{td}}$ and then noting that discrete groups of $G/G_{1,\text{td}}$ cannot have a subgroup containing $G^+_1/G^+_{1,\text{td}}$ as a limit point.

The remaining case is that $G_1$ is non-Archimedean and $G_2$ is Archimedean.
Let $V$ be a relatively compact identity neighborhood in the real Lie group $G_2$ such that $V$ has no non-trivial closed subgroups and moreover $H \cap (G_1 \times \overline{V} ) \subset G_1$.  Let $W$ be a symmetric identity neighborhood in $G_2$ satisfying $W^2 \subset V$. By Proposition \ref{prop:G has finitely genereted compact open} the non-Archimedean group $G_1$ admits a self-Chabauty-isolated compact open subgroup $P$. Choose open subsets $U_1,\ldots, U_n \subset P$ so that every closed subgroup of $P$ that intersects each $U_i$ non-trivially must be $P$ itself. Consider the Chabauty--open set $\Omega$
$$ 
 \Omega = \mathcal{O}_1\left( P \times (\overline{V}\setminus W)  \right) \cap \bigcap_{j=1}^{n} \mathcal{O}_2\left( U_j \times W \right). 
$$
On the one hand clearly $H \in \Omega$. On the other hand, observe that every closed subgroup $L  \in \Omega$  intersects non-trivially the subsets $U_j \times \{e_{G_2}\} $ for $j=1,\ldots,n$ and  satisfies  $ L \cap (P \times V) = P \times \{ e_{G_2} \} $. In particular, the Chabauty subspace of discrete subgroups in $G$ is disjoint from $\Omega$, as required.
\end{proof}

As a special case of Corollary \ref{cor:isolation of discrete from normal} we deduce that any non-discrete normal subgroup $N$ is Chabauty--isolated from the subspace of discrete subgroups in $G$.

\section{On Chabauty neighborhoods of lattices}
\label{sec:local rigidity}

We analyze the Chabauty neighborhoods of irreducible lattices in semisimple analytic groups. It turns out that in higher-rank there is a kind of local rigidity phenomenon, i.e. every irreducible lattice admits a Chabauty neighborhood consisting only of conjugates, see Theorem \ref{thm:Chabauty local rigidity for lattices in algebraic groups}. We develop a general principle allowing to deduce this from the classical notion of local rigidity.

The first step is to show that Chabauty neighborhoods of lattices in semisimple analytic groups are jointly discrete in the following sense.

\begin{definition}
	\label{def:jointly discrete}
	A  subgroup $\Delta \le H$ admits a \emph{jointly discrete Chabauty neighborhood} if there exist an identity neighborhood $V$ in $H$ and a Chabauty open neighborhood $\Omega$ of $\Delta$ in $\Sub{H}$ so that every closed subgroup $L \in \Omega$ satisfies $L \cap V = \{e\}$.
\end{definition}

\begin{thm}
	\label{thm:a lattice admits a jointly discrete Chabauty neighborhood - general case}
	Every lattice $\Gamma$  in a semisimple analytic group $G$  admits a jointly discrete Chabauty neighborhood.
\end{thm}

Note that the lattice $\Gamma$ is not assumed to be irreducible. In the following two subsections we deal with the non-Archimedean case of Theorem \ref{thm:a lattice admits a jointly discrete Chabauty neighborhood - general case}.

\subsection{Strongly regular hyperbolic elements}

Let $G$ be a non-Archimedean semisimple analytic group and $X$  the associated Bruhat--Tits building. 

\begin{definition}
\label{def:strongly regular hyperbolic elements}
An element $\gamma \in G$ is \emph{strongly regular hyperbolic} if $\gamma$ is hyperbolic and the two endpoints of any translation axis for $\gamma$ lie in the interiors of two opposite chambers  of the spherical building at infinity.
\end{definition}

The notion of strongly regular elements was introduced by Caprace and Ciobotaru  \cite{cap_cio}. This is a generalization of the notion of $\RR$-hyper-regular elements in semisimple real Lie groups \cite{prasad1972cartan}. 

The minimal set\footnote{The \emph{minimal set} $\mathrm{Min}(\gamma)$ of the element $\gamma$ is the set of points on which the minimal displacement of $\gamma$ is attained. $\mathrm{Min}(\gamma)$ is a closed convex subset of $X$. See \cite[II.6]{bh}.} of a strongly regular element of $G$ is a uniquely determined apartment in the building $X$ \cite[2.3]{cap_cio}. In particular, a strongly regular hyperbolic element belongs to a unique maximal $k$-split torus.


%

\begin{prop}
\label{prop:on bi-neighborhoods of strongly regular hyperbolic elements}
Let $\gamma \in G$ be a strongly regular hyperbolic element. Then there is an open identity neighborhood $U$ in $G$ and a number $N_\gamma \in \NN$ so that for $n \ge N_\gamma$ every element $\gamma' \in U \gamma^{n} U$  is strongly regular hyperbolic and satisfies $\mathrm{Min}(\gamma') = u \mathrm{Min}(\gamma)$ for some $u \in U$.
\end{prop}

\begin{proof}
	Let $\Delta$ be the unique apartment in the building $X$ with $\mathrm{Min}(\gamma) = \Delta$. Let $l : \RR\to X$ be a translation axis for $\gamma$ and assume that $x_0 = l(0)$ is a special vertex. So $l(\RR)$ is a geodesic line contained in the apartment $\Delta$ and crossing every wall\footnote{A \emph{wall} in $\Delta$ is the fixed point set of a reflection by an element in the affine Weyl group \cite[p. 34]{garrett1997buildings}.} \cite[2.2]{cap_cio}. We have $\gamma l(t) = l(t+d)$ for some $d > 0$ and all $t \in \RR$. 
	
	Choose $ N_\gamma \in \NN$ sufficiently large so that the two special vertices $ x_0$ and $\gamma^{n} x_0$ belong to the interior of opposite sectors\footnote{A \emph{sector} in $\Delta$ at the vertex $x_0$ is a connected component of the complement of all walls passing through $x_0$. It is a simplicial cone $x_0 + C$. The two sectors $x_0 + C$ and $x_0 - C$ are \emph{opposite} \cite[p. 221]{garrett1997buildings}} of $\Delta$ for all $ n \ge N_\gamma$. 

	Denote $I = \left[0,1\right]$ and let $U \le G$ be a compact open subgroup fixing the set $l(I)$ point-wise. Consider an element $\gamma' \in G$ of the form
	$$ 
	\gamma' = u_1 \gamma^n u_2, \quad u_1,u_2 \in U, n \ge  N_\gamma. 
	$$
	Denote $l_1 = u_1 \circ l$ so that $l_1(\RR)$ is a geodesic line contained in the apartment $\Delta_1 = u_1 \Delta$ satisfying $l(t) = l_1(t)$ for all $ t \in I$. Therefore
	$$ \gamma' l_1(t) = u_1 \gamma^n u_2 l_1(t) = u_1 \gamma^n l(t) = u_1 l (t + nd) = l_1(t+nd) $$
	for all $t \in I$. 
	
	This calculation implies that $\gamma'$ is a hyperbolic element admitting some translation axis $l' : \RR \to X$ containing the geodesic segment $l_1(\left[0,nd\right])$   \cite[2.8]{cap_cio}.
	The two special vertices $x_0$ and $\gamma' x_0 = u_1 \gamma^n x_0=l'(nd) = l_1(nd)$ lie along the translation axis $l'$ and belong to opposite sectors of the apartment $\Delta_1$. Relying on the local criterion of \cite[2.5]{cap_cio} we deduce that $\gamma'$ is strongly regular hyperbolic.
	
	The minimal set $\mathrm{Min}(\gamma')$ is equal to the unique apartment $\Delta'$ containing the axis $l'$. By the strong transitivity  property of buildings there is an element $u \in G$ satisfying
	$$ 
	u \Delta = \Delta'  \quad \text{and} \quad u_{|\Delta \cap \Delta'} = \mathrm{id}.
	$$
	The geodesic segment $l(I) = l_1(I) = l'(I)$ is contained in $\Delta \cap \Delta'$ and therefore $u \in U$. The proof is complete as $\Delta = \mathrm{Min}(\gamma)$ and $\Delta' = \mathrm{Min}(\gamma')$.
\end{proof}

\begin{prop}
\label{prop:a lattice contains a strongly regular hyperbolic element}
Every lattice $\Gamma $ in $ G$ contains strongly regular hyperbolic elements.
\end{prop}
\begin{proof}
The group $G$ contains strongly regular hyperbolic elements according to Theorem 1.2 of \cite{cap_cio}. It is now standard to deduce that $\Gamma$ has strongly regular hyperbolic elements as well, relying on  Proposition \ref{prop:on bi-neighborhoods of strongly regular hyperbolic elements}. In fact, we only need to use the fact that $\Gamma$ has the so called property (S)  \cite[5.1, 5.4]{raghun}.
\end{proof}

\subsection{Joint discreteness in the non-Achimedean case}

We are ready to prove the following non-Archimedean case of Theorem \ref{thm:a lattice admits a jointly discrete Chabauty neighborhood - general case}. Recall that jointly discrete Cahabauty neighborhoods were introduced in Definition \ref{def:jointly discrete} above.

\begin{thm} 
	\label{thm:a lattice admits a jointly discrete Chabauty neighborhood}
	Every lattice $\Gamma$  in a non-Archimedean semisimple analytic group $G$ admits a jointly discrete Chabauty neighborhood.
\end{thm}

We will make use of the following lemma.

\begin{lemma}

	\label{lem:on intersections of parabolics}
	Let $S $ be a maximal $k$-split torus  and $ \Gamma $ any  Zariski-dense subgroup in $G$. 
Then there are elements $\gamma_1, \ldots, \gamma_m\in \Gamma$ for some $ m \in \NN$ and an identity neighborhood $W \subset G$ so that  the intersection $\bigcap_{i=1}^{m}S^{\gamma_i w_i } $ is central in $G$ for every choice of  elements $w_i \in W$.
\end{lemma}

\begin{proof}
Consider the intersection $S_0 = \bigcap_{\gamma \in \Gamma} S^\gamma$ as well as the group $\tilde S=\langle S^\gamma:\gamma \in \Gamma\rangle$. Since $\Gamma$ is Zariski dense, the Zariski closure of $\tilde S$ is normal in $G$, and as it contains a maximal split torus it must be all of $G$. Since $S_0$ is central in $\tilde S$, it is also central in $G$. 
 	The conclusion now follows from Proposition \ref{prop:intersections over deformed elements}.
\end{proof}



\begin{proof}[Proof of Theorem \ref{thm:a lattice admits a jointly discrete Chabauty neighborhood}]
Let $G$ be a non-Archimedean semisimple analytic group, $X$ the corresponding Bruht--Tits building and $\Gamma$ a lattice in $G$.
	
The lattice $\Gamma$ contains a strongly regular hyperbolic element $\gamma$ by Proposition \ref{prop:a lattice contains a strongly regular hyperbolic element}. The element $\gamma$ preserves the apartment $\Delta = \mathrm{Min}(\gamma)$ in $X$. Let $S $ be the unique maximal $k$-split torus in $G$ containing $\gamma$. Note that $S$ preserves $\Delta$  as well.

Making use of Lemma \ref{lem:on intersections of parabolics} we find elements $\gamma_1, \ldots, \gamma_m \in \Gamma$ for some $m \in \NN$  and an identity neighborhood $W \subset G$ such that
$ \bigcap_{i=1}^{m} S^{\gamma'_i}$ is central in $G$ for all $\gamma'_i \in \gamma_i W$. 
Every conjugate $\gamma^{\gamma_i} = \gamma_i^{-1} \gamma \gamma_i$ is strongly regular hyperbolic as well. Let $U$  be an identity neighborhood contained in $W$ and $N \in \NN$ be such that every element $\lambda_i \in U(\gamma^N)^{\gamma_i} U$   is strongly regular hyperbolic, as in Proposition \ref{prop:on bi-neighborhoods of strongly regular hyperbolic elements}. The unique $k$-split torus containing $(\gamma^N)^{\gamma_i}$ is still $S^\gamma_i$ and the second part of Proposition \ref{prop:on bi-neighborhoods of strongly regular hyperbolic elements} shows that $\lambda_i \in S^{\gamma_i u_i}$ for some $u_i \in U \subset W$. 

Fix an additional relatively compact identity neighborhood $V$ in $G$. Relying on Lemma 9.4 of \cite{GL} we may construct  a Chabauty neighborhood $\Omega$ of $\Gamma$ in $\Sub{G}$ with the following properties ---  every closed subgroup $L \in \Omega$ satisfies that
\begin{itemize}
	\item $L_V  = L \cap V$ is a subgroup of $G$,
	\item $L$ contains elements $\lambda_i \in U (\gamma^N)^{ \gamma_i} U$ for all $ i =1, \ldots, m$,  
	\item the subgroup $L_V$ is normalized by  $ L_0 = \left< \lambda_1, \ldots, \lambda_m \right> \le L$, and
	\item $L_V \cap Z(G) = \{e\}$ where $Z(G)$ is the center of $G$.
\end{itemize}
The last of the above requirements does not appear in \cite{GL}. However  $Z(G)$  is finite and it is clear that we may add this additional requirement on $\Omega$.

We claim that for every closed subgroup $L \in \Omega$ the intersection  $L_V = L \cap V$ consists of the identity element.  
Consider some $L \in \Omega$ and let $F \subset X$ denote the closed convex subset of $L_V$-fixed points. The compactness of the subgroup $L_V$ implies that  $F$ is  nonempty. Clearly $F$ is $L_0$-invariant.

By construction the elements $\lambda_i$ are strongly regular hyperbolic. Since $F$ is $\lambda_i$-invariant its boundary  $\partial F $  contains the two endpoints $l_i(\infty)$ and $l_i(-\infty)$  of some translation axis $l_i$ for $\lambda_i$. In particular $\partial F$ is nonempty. 
Since  $L_V$ fixes $F$ it must also  fix the two endpoints  $l_i(\infty)$ and $l_i(-\infty)$.  These endpoints belong to the interiors of opposite chambers at infinity and $L_V$ must preserve the unique apartment determined by this pair. Therefore $L_V$ is contained in the torus $S^{\gamma_i u_i}$. To conclude
$$ 
 L_V \le L_V \cap \bigcap_{i=1}^{m}S^{\gamma_i u_i} \le  L_V \cap Z(G) = \{ e \}. 
$$
We have shown that $\Gamma$ admits a jointly discrete Chabauty neighborhood $\Omega$.
\end{proof}

\subsection{Joint discreteness in the general case}

We extend Chabauty joint discreteness to general semisimple analytic groups, making use of the fact that Lie groups have no small subgroups.

\begin{proof}[Proof of Theorem \ref{thm:a lattice admits a jointly discrete Chabauty neighborhood - general case}]
Write $G$ as an almost direct product $G = C \times D$ where $C$ and $D$  are  connected and  totally disconnected semisimple analytic groups, respectively. Let $\Gamma$ be a lattice in $G$.  The proof of Theorem \ref{thm:a lattice admits a jointly discrete Chabauty neighborhood} given above can be easily adapted to this case, as follows.

Let $V_C \subset C$ be an identity neighborhood which does not contain any closed non-trivial subgroup. Let $V_D \le D$ be a relatively compact identity neighborhood as in the proof of Theorem \ref{thm:a lattice admits a jointly discrete Chabauty neighborhood}. Take $V = V_C \times V_D$ to be used in the definition of the Chabauty neighborhood $\Omega$. Note that every closed subgroup $L \in \Omega$ satisfies $\overline{\mathrm{pr}_C (L \cap V)} = \{e_C\}$, so that $L_V = L \cap V$ can be regarded as a subgroup of the non-Archimedean factor $D$.

The argument of Proposition \ref{prop:a lattice contains a strongly regular hyperbolic element} implies that $\Gamma$ admits an element $\gamma$ whose projection to $D$ is strongly regular hyperbolic. The rest of the proof of Theorem \ref{thm:a lattice admits a jointly discrete Chabauty neighborhood} applies with no changes.
\end{proof}

\subsection{Local rigidity and Chabauty neighborhoods}

Let $G$ be a topological group and $\Gamma \le G$ a lattice. Let $\mathrm{Hom}(\Gamma,G)$ denote the space of all homomorphisms from $\Gamma$ to $G$ with the point-wise convergence topology. If $\Sigma$ is a generating set for $\Gamma$ then $\mathrm{Hom}(\Gamma,G)$ can be identified with a closed subset of $G^\Sigma$.
The lattice $\Gamma $ is \emph{weakly locally rigid} if the inclusion morphism $\iota : \Gamma \to G$ admits an open neighborhood in $\mathrm{Hom}(\Gamma,G)$ consisting of conjugates by an  automorphism of $G$. $\Gamma$ is \emph{locally rigid} if this automorphism is inner  \cite[Sec. 2]{GL}.


\begin{definition}[Relative finite presentability]
\label{def:finitely presented inside}
$\Gamma$ is \emph{finitely presented inside} $G$ if it has a finite generating set $\Sigma$ and there are finitely many words $\omega_1,\ldots,\omega_n \in F_\Sigma$ so that 
$$ \mathrm{Hom}(\Gamma,G) = \{ (g_\sigma) \in G^\Sigma \: : \: \omega_i(g_\sigma) = e_G \;\; \forall i \in \{1,\ldots,n\} \} $$
Here $F_\Sigma$ in the free group on the generators $\Sigma$.
\end{definition}

This notion is clearly independent of the choice of a particular finite generating set for $\Gamma$. 
If $\Gamma$ itself is finitely presented then it is finitely presented inside any group. Moreover any finitely generated group is finitely presented inside an algebraic group by Noetherianity.

We are ready to formulate a general principle allowing to promote classical local rigidity to its the Chabauty topology variant.

\begin{prop}
\label{prop:local rigidity implies Chabauty local rigidity}
Let $G$ be a  locally compact group whose automorphism group preserves the Haar measure. Let $\Gamma $ be a weakly locally rigid lattice in $G$ which is finitely presented inside $G$ and admits a jointly discrete Chabauty neighborhood.

Then $\Gamma$ has a Chabauty neighborhood consisting of conjugates by an automorphism of $G$. If $\Gamma$ is moreover locally rigid then this automorphism is inner.
\end{prop}


\begin{proof}
Let $\Sigma$ be a finite generating set for $\Gamma$ and $\omega_1,\ldots,\omega_n \in F_\Sigma$ words\footnote{If $G$ is an algebraic group then these words $\omega_i$ define $\mathrm{Hom}(\Gamma,G)$ as an algebraic variety.} exhibiting $\Gamma$ to be finitely presented inside $G$ as in Definition \ref{def:finitely presented inside}. Let $m=\max_{i=1,\ldots,n}|\omega_i|$ denote the maximal length of the words $\omega_i$.

Since $\Gamma$ is weakly locally rigid there is an identity neighborhood $U$ in $G$ so that every group homomorphism $f : \Gamma \to G$ with $f(\sigma) \in U \sigma$ for every $\sigma \in \Sigma$ is given by
$$ f(\gamma) = \alpha(\gamma), \quad \forall \gamma \in \Gamma $$
for some  automorphism $\alpha$ of $G$. If $\Gamma$ is  assumed to be locally rigid then $\alpha$ is inner.

Joint discreteness allows us to choose an identity neighborhood $V$ in $G$ and an open Chabauty neighborhood $\Omega_1 \subset \Sub{G}$ of $\Gamma$ so that every closed subgroup $L \in \Omega_1$ satisfies $L\cap V = \{e\}$. Denote $\Sigma_0 = \Sigma \cup \{e\}$ and take an   identity neighborhood $W$ in  $G$ satisfying
$$ 
 W \subset U \quad \text{and} \quad \left(W^{(\Sigma_0^{m-1})}\right)^m \subset V.
$$

Let  $\Omega_2 \subset \Sub{G}$ be a Chabauty neighborhood of $\Gamma$ so that every closed subgroup $L \in \Omega_2$ contains elements $l_\sigma \in W\sigma$ for every $\sigma \in \Sigma$.
The argument of \cite[p. 28]{raghun} shows that there is a Chabauty neighborhood $\Omega_3 \subset \Sub{G}$ of $\Gamma$ so that every $L \in \Omega_1 \cap \Omega_3$ satisfies 
$ \text{co-vol}(L) > \frac{1}{2}\text{co-vol}(\Gamma) $.

Consider a closed subgroup $L \le G$ with $L \in \Omega_1 \cap \Omega_2 \cap \Omega_3$. There are elements $l_\sigma \in L \cap W\sigma$ for $\sigma \in \Sigma$ as above. We claim that  the homomorphism $\varphi$ given by
$$ \varphi : F_\Sigma \to L, \quad \varphi(\sigma) = l_\sigma \quad \forall \sigma \in \Sigma$$
represents an element of $\mathrm{Hom}(\Gamma,G)$. 
Write $\omega_i = \sigma_{1}\cdots\sigma_{m_i}$ with $\sigma_j \in \Sigma$. 
Then
$$ 
 \varphi(\omega_i) \in W\sigma_1W\sigma_2W\cdots W\sigma_{m_i} \subset W  W^{\sigma_1} W^{\sigma_1 \sigma_2} \cdots W^{\sigma_1\cdots \sigma_{m_i-1}} \omega_i. 
$$
However the word $\omega_i = \sigma_1 \cdots \sigma_{m_i}$ evaluated in $G$ is equal to the identity. Therefore
$$ 
 \varphi(\omega_i) \in L \cap \left(W^{(\Sigma_0)^{m-1}}\right)^{m} \subset L \cap V =  \{e_G\}.
$$

Since $\Gamma$ is presented inside $G$ by $\{\omega_i,~i=1,\ldots,n\}$ we obtain  a well-defined homomorphism $\psi \in \mathrm{Hom}(\Gamma,G)$  by taking
$$ \psi : \Gamma \to G, \quad \psi(\sigma) = l_\sigma \quad \forall \sigma \in \Sigma$$

By the above construction and as $\Gamma$ is weakly locally rigid there is an automorphism $\alpha$ of $G$ so that $\psi(\gamma) = \alpha(\gamma)$ for all $\gamma \in \Gamma$.
The group $L$ is discrete and contains $\alpha(\Gamma)$ so that  $L$  is a lattice. By assumption the automorphism $\alpha$ preserves Haar measure so that $\text{co-vol}(\alpha(\Gamma)) = \text{co-vol}(\Gamma)$. We obtain 
$$ 
 2 > \text{co-vol}(\Gamma) / \text{co-vol}(L) = \text{co-vol}(\alpha(\Gamma)) / \text{co-vol}(L) = \left[ L : \psi(\Gamma) \right] \in \NN.
$$
This can only be the case provided that $\psi(\Gamma) = \alpha(\Gamma) = L$, as required. 

If $\Gamma$ is moreover assumed to be locally rigid then $\alpha$ is inner. In that case  $\alpha$ clearly preserves the Haar measure on $G$ and Proposition \ref{prop:outer automorphism group preserves Haar measure}  becomes redundant.
\end{proof}


\subsection{Chabauty Local rigidity in semisimple analytic groups}

We complete the proof that an irreducible lattice in a higher rank semisimple analytic group admits a Chabauty neighborhood consisting of conjugates.

\begin{proof}[Proof of Theorem \ref{thm:Chabauty local rigidity for lattices in algebraic groups}]

Let $G$ be a semisimple analytic group with $\mathrm{rank}(G) \ge 2$   and $\Gamma \le G$ an irreducible lattice. The  conclusion of Theorem  \ref{thm:Chabauty local rigidity for lattices in algebraic groups} will follow from Proposition \ref{prop:local rigidity implies Chabauty local rigidity} as soon as  we verify all of the required assumptions.

The automorphism group of $G$ preserves Haar measure  by Proposition \ref{prop:outer automorphism group preserves Haar measure}.
The lattice $\Gamma$ is finitely generated \cite[IX.3]{Ma} and finitely presented inside $G$  in the sense of Definition \ref{def:finitely presented inside} by Noetherianity.
The fact that $\Gamma$ admits a jointly discrete Chabauty neighborhood is established in Theorem \ref{thm:a lattice admits a jointly discrete Chabauty neighborhood - general case}.
	
It follows from Margulis super-rigidity theorem that $\Gamma$ is weakly locally rigid, and that $\Gamma$ is locally rigid provided that $\mathrm{char}(k) = 0$  \cite[p. 241, Thms. A,C]{Ma}. 
\end{proof}

\subsection{Wang's finiteness in the non-Archimedean case}

We   deduce Wang's finiteness relying on Chabauty local rigidity. This is essentially a standard argument appearing already in  \cite{wang}.

\begin{proof}[Proof of Theorem \ref{thm:Wang}]
Let $G$ be a semisimple analytic group with $\mathrm{rank}(G) \ge 2$.  We may assume that all of the simple  factors of $G$ are of the same characteristic, for otherwise $G$ admits no irreducible lattices.

We claim that $G$ has the Kazhdan--Margulis property for lattices \cite{kazhdan1968proof,gel_kaz,GL}. This means that there is an identity neighborhood $U$ in $G$ such that every lattice $\Gamma$ in $G$ admits a conjugate intersecting $U$ only at the identity. If $G$ is Archimedean then this is a classical result --- see \cite{kazhdan1968proof} and also \cite{gel_kaz}. If $G$ is zero characteristic non--Archimedean   then  much more is true, namely all lattices in $G$ are jointly discrete \cite[LG 4.27, Th. 5]{sere}. If $G$ is simply-connected and  positive characteristic non-Archimedean    then $G$ has the Kazhdan-Margulis property by Theorem A of \cite{golsefidy2009lattices}. 
In the remaining case $G = C \times D$ where  $C$ is Archimedean and $D$ is zero characteristic non-Archimedean semisimple analytic groups. Let $K \le D$ be any compact open subgroup.  For any lattice $\Gamma$ in $G$  the projection of $\Gamma \cap (C\times K)$ to $C$ is a lattice.  Therefore if $U$ is a Kazhdan--Margulis neighborhood in $C$ then so is $U \times K$ in $G$.

Fix a co-volume $ v > 0$. Let $\Omega_v $ denote the Chabauty subspace of $\Sub{G}$ consisting of lattices of co-volume bounded above by $v$ and intersectiong the Kazhdan--Margulis neighborhood of $G$ only at the identity. The Chabauty criterion \cite[1.20]{raghun} implies that $\Omega_v$ is compact. We conclude the proof of Wang's finiteness relying on Chabauty local rigidity as in Theorem \ref{thm:Chabauty local rigidity for lattices in algebraic groups}.
\end{proof}

\section{Accumulation points of invariant random subgroups}
\label{sec:accumulation points of invariant random subgroups}

We prove our main result on accumulation points of lattices in higher rank semisimple analytic groups with property $(T)$, generalizing the analog result of \cite{7S} from the classical Archimedean case to the general one.

Just as in \cite{7S} the proof relies on the celebrated rigidity theorem of Stuck and Zimmer \cite{sz}. This  theorem essentially provides classifies the ergodic invariant random subgroups in the situation under consideration. In our general setup we shall also require the non-Archimedean version of this theorem which was established in \cite{levit}, as well as the results obtained in \S \ref{sec:isolated groups} and \S \ref{sec:local rigidity}.


%

\begin{proof}[Proof of Theorem \ref{thm:accumulation points of invariant random subgroups}]
Let $G$ be an happy semisimple analytic group with $\mathrm{rank}(G) \ge 2$ and property $(T)$. Let $\Gamma_n$ be a sequence of  irreducible lattices in $G$ and assume that the $\Gamma_n$'s  are pairwise non-conjugate in the appropriate sense. By this we mean that the $\Gamma_n$'s are pairwise non-conjugate in zero characteristic and pairwise non-conjugate by an automorphism of $G$ in positive characteristic \footnote{In view of the Borel--Prasad finiteness theorem \cite{bp} we have in our situation that $\text{vol}(G/\Gamma_n)\to\infty$. This fact however is not needed for the proof. }.

Let $\mu_n$ denote the invariant random subgroup corresponding to the lattice $\Gamma_n$ and consider a weak-$*$ accumulation point $\mu \in \IRS{G}$ of the sequence $\mu_n$. The proof consists in showing that $\mu$ must be $  \delta_Z$ for some central subgroup $Z$ in $G$. Our strategy is to proceed by elimination.
	
Since every factor of the group $G$ has property $(T)$ the accumulation point $\mu$ is irreducible by a result of Glasner--Weiss \cite{glasner1997kazhdan}. The Stuck--Zimmer rigidity theorem \cite{sz, levit} applies in this situation. This means that the $G$-action on $(\Sub{G},\mu)$ is either essentially transitive or has central stabilizers. The stabilizer of each point $H \in \Sub{G}$ is the normalizer $N_G(H)$ hence the second possibility is ruled out.

We may assume that $\mu$-almost every subgroup is conjugate to some fixed closed subgroup $H \le G$. In particular $(\Sub{G},\mu)$ is isomorphic as a $G$-space to the homogeneous space $G/N_G(H)$. This homogeneous space supports a $G$-invariant probability measure. The main result of  \cite{margulis1977cobounded} therefore implies that 
$$N_G(H) = M \times \Gamma$$
so that $G$ can be written as an almost direct product $G = G_1 \times G_2$ with $G_1^+ \le M \le G_1$ and $\Gamma$  being a lattice in $G_2$. Examining the various possibilities for $H$ shows that it must be of the form $H = N \times \Delta$ where $N \nrm M$ and $\Delta \nrm \Gamma$ is a normal subgroup with $N_{G_2}(\Delta) = \Gamma$.

We claim that $G_1$ must be either trivial or equal to $G$. For if $G_1$ is a non-trivial proper subgroup of $G$ then $G_1^+$ lies in the kernel of the $G$-action on the non-atomic homogeneous space $G/(M\times \Gamma)$ which is a contradiction to irreducibility. 

If $G_1$ is trivial then $\Gamma$ is an irreducible lattice in $G$ and $\Delta$ is a normal subgroup of $\Gamma$. The normal subgroup theorem of Margulis \cite[Chapter IV]{Ma} implies that $\Delta$ is either a lattice or central in $G$. In the current situation $\Delta$ must indeed be an irreducible lattice  since $N_G(\Delta) = \Gamma$. However $\mu_\Delta$ is an isolated point in the extreme points of $\IRS{G}$ by Corollary \ref{cor:non conjugate lattices are discrete IRS} and we arrive at a contradiction.

If $G_1$ is equal to $G$ then $H = N \nrm G$ is a normal subgroup. Since $\mu$ is a limit of invariant random subgroups supported on discrete subgroups, Corollary \ref{cor:isolation of discrete from normal} implies that $H$ must be discrete and in particular central in $G$.
%
%
%
%
%
%
%
%
%
\end{proof}

\section{Limit formula for normalized Betti numbers}

We prove a limit formula for the normalized Betti numbers as in Theorem \ref{thm:limit of normalized betti numbers is l2-betti number}. Note that the next and final section \S \ref{sec:plancherel} contains a much more general result concerning the limit of normalized relative Plancherel measures in the zero characteristic case. The argument of current section however is straightforward relying on a beautiful result of Elek and L\"{u}ck's approximation theorem \cite{luck}, and it applies in arbitrary characteristic.

Let $b_d(\Sigma)$ denote the $d$-th Betti number of a given  simplicial complex $\Sigma$. We make use of the following general result due to Elek \cite[Lemma 6.1]{elek}.

\begin{theorem}[Elek]
\label{thm:limit of normalized betti numbers exists}
Let $\Sigma_i$ be a Benjamini--Schramm convergent sequence of finite simplicial complexes with uniformly bounded degree. Then  the limit of
$  \frac{b_d(\Sigma_i)}{\left|V(\Sigma_i) \right|} $ exists as $i \to \infty$ for every $d \in \NN \cup \{0\}$.
\end{theorem}

Note that a specific Benjamini--Schramm limit is not assumed for the sequence $\Sigma_i$ and that the limit of the normalized Betti numbers is not specified.


\begin{proof}[Proof of Theorem \ref{thm:limit of normalized betti numbers is l2-betti number}]
Let $G$ be a non-Archimedean semisimple analytic group and $X$ the associated Bruhat--Tits building.
Let $\Gamma_i$ be a sequence of torsion-free uniform lattices in $G$ so that the corresponding invariant random subgroups $\mu_{\Gamma_i}$ weak-$*$ converge to $\delta_{\{e\}}$. 

Let $\Delta\le G$ be a fixed arbitrary uniform torsion-free lattice.
The lattice $\Delta$ is residually finite.  Therefore $\Delta$ has a descending sequence of normal subgroups $\Delta_i \nrm \Delta$ with $\bigcap_{i \in \NN} \Delta_i = \{e\}$. By L\"{u}ck's approximation theorem 
$$b_d^{(2)}(\Delta) = \lim_{i\to\infty} \frac{b_d(\Delta_i)}{\left[\Delta:\Delta_i\right]} \quad \forall d \in \NN \cup \{0\}.
$$

Consider a sequence of lattices alternating between the $\Gamma_i$'s at the odd positions and $\Delta_i$'s at the even positions. The assumption on the $\Gamma_i$'s together with Example \ref{prop:invariant random subgroups of a descending sequence go to delta e}   imply that the associated invariant random subgroups converge to $\delta_{\{e\}}$. The associated alternating sequence $\Sigma_i$ of simplicial complexes is given by
$$ 
 \Sigma_i = \begin{cases} \Gamma_{j} \backslash X, & i = 2j-1 \\ \Delta_{j}\backslash X, & i = 2j  \end{cases}
$$

Observe that $\Sigma_i$ is a convergent sequence in the Benjamini--Schramm sense by Corollary \ref{cor:map from IRS to BS is continuous}. Elek's theorem  implies that the limit of $\frac{b_d(\Sigma_i)}{\abs{V(\Sigma_i)}}$ exists for every $d \in \NN\cup\{0\}$ and as $i \to \infty$. Restricting to the even positions we get
$$ 
 \lim_{j\to\infty} \frac{b_d(\Sigma_{2j})}{\abs{V(\Sigma_{2j})}} = \frac{1}{\abs{V(X/\Delta)}}\lim_{j\to\infty} \frac{b_d(\Delta_j)}{\left[\Delta :  \Delta_j \right]} =\frac{\mu_G(K)}{\kappa_G}    \frac{b_d^{(2)}(\Delta)}{\text{vol}(G/\Delta)}, \quad \forall d \in \NN \cup \{0\}.
$$
Here $\mu_G$ is the Haar measure on $G$, $K$ is a maximal compact subgroup in $G$ and $\kappa_G$ is the number of vertices in a fundamental domain for the $G$-action on $X$.

The limit of the normalized Betti numbers is the same along the odd subsequence $\Sigma_{2j-1}$ associated to the lattices $\Gamma_j$. Therefore
$$ 
 \lim_{j\to\infty} \frac{b_d(\Gamma_j)}{\text{vol}(G/\Gamma_j)} = \frac{\kappa_G}{ \mu_G(K)}\lim_{j\to\infty} \frac{b_d(\Sigma_{2j-1})} {\abs{V(\Sigma_{2j-1})}} =  \frac{b_d^{(2)}(\Delta)}{\text{vol}(G/\Delta)}, \quad \forall d \in \NN \cup \{0\}. 
$$
The right hand side is equal to $b_d^{(2)}(G)$ for every $d \in \NN \cup \{0\} $, as required.
\end{proof}

\section{Convergence of relative Plancherel measures}
\label{sec:plancherel}

Let $G$ be a semisimple analytic group over zero characteristic local fields. 
We prove Theorem \ref{thm:BS convergence implies plancherel convergence} about convergence of normalized relative Plancherel measures. 


\subsection{Plancherel measure and the Sauvageot principle}

Let $\widehat{G}$ denote the \emph{unitary dual} of $G$, i.e. the set of all equivalence classes of unitary irreducible representations of $G$. The unitary dual $\widehat{G}$ is a topological space with the Fell topology. 
The \emph{Plancherel measure} $\nu^G$ on $\widehat{G}$ is associated to the right regular representation of $G$ in $L^2(G, \mu_G)$. Note that $\nu^G$ depends the choice of Haar measure up to scaling.

Let $\varphi \in C^\infty_c(G)$ be a smooth function\footnote{A function $\varphi$ on the semisimple analytic group $G$ is \emph{smooth} if it locally a product of a $C^\infty$-function and a constant function  on the Archimedean and non-Archimedean factors,  respectively.} with compact support.  
For every irreducible representation $\pi \in \widehat{G}$ the operator $\pi(\varphi)$ on $L^2(G,\mu_G)$ is trace class.
The Fourier transform $\widehat{\varphi}$ is the function on $\widehat{G}$ given by
$$ 
 \widehat{\varphi} : \pi \mapsto \mathrm{trace} \; \pi(\varphi).
$$

\begin{theorem}[Plancherel  formula]
\label{thm:Plancherel inversion}
$ \nu^G(\widehat{\varphi}) = \varphi(e)$  for every $\varphi \in C_c^\infty(G)$.
\end{theorem}
\begin{proof}
This is the Fourier inversion formula, being  a consequence of the Plancherel formula. It holds true for unimodular l.c.s.c. groups of type I \cite[ 7.44]{folland2016course}.

A semisimple analytic group is type I by \cite{harish_chandra_type_I} and \cite{bernstein1974all} in the Archimedean and non-Archimedean cases, respectively.  The property of being type I is preserved under taking a  direct product of finitely many factors \cite[7.25]{folland2016course}.
\end{proof}

\begin{theorem}[Sauvageot density principle]
\label{thm:Sauvageot density principle}
Let $\nu_n$ be a sequence of measures on the unitary dual $\widehat{G}$. If $\nu_n(\widehat{\varphi}) \to \nu^G(\widehat{\varphi})$ for every $\varphi \in C_c^\infty(G)$ then  $\nu_n(E) \to \nu^G(E)$ for every  relatively quasi-compact $\nu^G$-regular subset $E \subset \widehat{G}$.
\end{theorem}
\begin{proof}
See Theorem 7.3 and Corollary 7.4 of Sauvageot's work \cite{sauvageot1997principe}. The subset $E$ is assumed to be open in \cite{sauvageot1997principe}, however Shin observes \cite{shin2012automorphic} that this assumption can be removed.
\end{proof}

\subsection{Uniform lattices and relative Plancherel measure}
Let $\Gamma$ be a uniform lattice in $G$. Consider a smooth compactly supported function $\varphi \in C_c^\infty(G)$ so that $\rho_\Gamma(\varphi)$ is a compact operator on the Hilbert space  $L^2(\Gamma\backslash G, \mu_G)$ \cite[\S 1.2.2]{gelfand1969representation}. Moreover
$$ 
 \textrm{trace} \; \rho_\Gamma(\varphi) = 
 \sum_{\pi \in \widehat{G}} m(\pi,\Gamma) \textrm{trace} \; \pi(\varphi) = 
 \mathrm{vol}(\Gamma \backslash G) \nu_{\Gamma} (\widehat{\varphi}),
$$
where $\nu_\Gamma$ is the relative Plancherel measure of $G$ with respect to the lattice $\Gamma$, as introduced in Definition \ref{def:relative plancherel} above.

Let $\mathrm{Sub}_d(G) $ denote the Chabauty subset of all the discrete subgroups in $G$.
Given a function $\varphi \in C_c^\infty(G)$ we define a map $ \Theta_\varphi$ 
$$  \Theta_\varphi : \mathrm{Sub}_d(G) \to \RR,  \quad \Theta_\varphi(\Lambda) = \sum_{\lambda \in \Lambda} \varphi(\lambda)   $$
for every discrete subgroup $\Lambda$ in $G$ discrete. In particular
$$ 
 \Theta_\varphi(\{e\}) = \varphi(e). 
$$

Given an identity neighborhood $U \subset G$ let $\mathrm{Sub}_{d,U}(G)$ denote the subset of $\mathrm{Sub}_d(G)$ consisting of those subgroups $\Delta$ satisfying $\Delta \cap U = \{e\}$.

\begin{prop}
\label{prop:Theta admits a continuous extension}
For  every $\varphi \in C_c^\infty(G)$ the map $\Theta_\varphi$ is continuous on $\mathrm{Sub}_{d,U}(G)$ and  admits a continuous extension to all of $\mathrm{Sub}(G)$.
\end{prop}
\begin{proof}
Write $G = C \times D$ where $C$ and $D$ are the Archimedean and zero characteristic non-Archimedean semisimple analytic factors of $G$, respectively.  $D$ admits a compact open subgroup $V_D$ such that every discrete subgroup in $D$ interests $V_D$ trivially \cite[LG 4.27, Th. 5]{sere} and $C$ admits an relatively compact identity neighborhood $V_C$ without small closed subgroups and so that $V_C^2 \times V_D \subset U$. Denote $V = V_C \times V_D$.


Fix a compactly supported smooth function  $\varphi \in C_c^\infty(G)$ and denote $Q = \mathrm{supp}(\varphi)$ so that $Q$ is compact. Cover $Q$ by finitely many translates $V_1,\ldots, V_n$ of $V$ for some $n \in \NN$, with $V_1=V$. Note that every closed subgroup belonging to $\mathrm{Sub}_{d,U}(G)$ intersects each translate $V_i$ in at most one point. 

We first show  that $\Theta_\varphi$ is continuous on $\mathrm{Sub}_{d,U}(G)$. Let $\Delta$ be any subgroup of $G$ with $\Delta \cap U = \{e\}$. 
Assume without loss of generality that $\Delta \cap V_i = \{v_i\} $  for all $ 1 \le i \le m$ with some $1 \le m \le n$ and $\Delta \cap V_i = \emptyset $ for all $ m < i \le n$. Up to shrinking the neighborhood $V$  if necessary, we may moreover assume that in fact $\Delta \cap \overline{V}_i  = \emptyset$ for all $ m < i \le n$. Taking $W$ to be an sufficiently small identity neighborhood in $G$,  the value of $\Theta_\varphi$ on the Chabauty-open neighborhood
$$ \Omega = \bigcap_{i=1}^m \mathcal{O}_2(Wv_i) \cap \bigcap_{i>m}^n \mathcal{O}_1(\overline{V}_i)  $$
of $\Delta$ can be made arbitrary close to $\Theta_\varphi(\Delta)$. So that $\Theta_\varphi$ is continuous on $\mathrm{Sub}_{d,U}$ and is moreover bounded in absolute value by $n \norm{\varphi}_\infty$.

Note that $\mathrm{Sub}_{d,U}(G)$  is a Chabauty-closed subset of $\mathrm{Sub}(G)$, as  its complement is nothing but the Chabauty-open subset $\mathcal{O}_2(U \setminus \{e\})$.
Therefore $\Theta_\varphi$ can be extended to a continuous function   all of $\mathrm{Sub}(G)$ by Tietze's extension theorem.
\end{proof}


\begin{prop}
\label{prop:a computation of trace in terms of kernel}
Let $\Gamma$ be a uniform lattice in $G$. Let  $\mu_\Gamma$ be the invariant random subgroup and $\nu_\Gamma$ the relative Plancherel measure on $\widehat{G}$ associated to $\Gamma$.  Then
$ \nu_{\Gamma} (\widehat{\varphi}) = \mu_\Gamma(\Theta_\varphi) $ for every function $\varphi \in C_c^\infty(G)$.
\end{prop}
\begin{proof}
In view of Proposition \ref{prop:Theta admits a continuous extension} we may extend $\Theta_\varphi$ to a Chabauty-continuous function on $\textrm{Sub}(G)$.
The equality $ \nu_{\Gamma} (\widehat{\varphi}) = \mu_\Gamma(\Theta_\varphi) $ follows from the computation on p. 37-38 of \cite{7S}. This argument of \cite{7S} is general and applies to uniform lattices in any l.c.s.c. group.
\end{proof}

\subsection{Proof of convergence for relative Plancherel measures}

We conclude this section by providing proofs of Theorem \ref{thm:BS convergence implies plancherel convergence} and Corollary \ref{cor:BS convergence implies ptwise plancherel convergence}.

Let $G$ be a  semisimple analytic group in zero characteristic. Let $\Gamma_n$  be a uniformly discrete sequence of lattices in $G$ so that the corresponding invariant random subgroups $\mu_{\Gamma_n}$ converge to $\delta_{\{e\}}$. Let $\nu^G$ and $\nu_n$ denote  the Plancherel measure of $G$ and the relative normalized Placherel measures associated to the $\Gamma_n$'s, respectively.

\begin{proof}[Proof of Theorem \ref{thm:BS convergence implies plancherel convergence}]
In view of the Sauvageot density principle to prove Theorem \ref{thm:BS convergence implies plancherel convergence} we need to show that $\nu_n(\widehat{\varphi}) \to \nu^G(\widehat{\varphi})$ for every function $\varphi \in C_c^\infty(G)$.
Consider a fixed function $\varphi \in C_c^\infty(G)$ and regard $\Theta_\varphi$ as a continuous function on $\mathrm{Sub}(G)$ making use of Proposition \ref{prop:Theta admits a continuous extension}. In light of the Plancherel formula, Proposition \ref{prop:a computation of trace in terms of kernel}, and the   weak-$*$ convergence of $\mu_{\Gamma_n}$ to $\delta_{\{e\}}$  we obtain that
$$ \nu_{n}(\widehat{\varphi}) = \mu_{n}(\Theta_\varphi)  \xrightarrow{n \to \infty} \delta_{\{e\}}(\Theta_\varphi) = \Theta_\varphi(\{e\}) = \varphi(e) = \nu^G(\widehat{\varphi}) $$
 as required.
 \end{proof}

\begin{proof}[Proof of Corollary \ref{cor:BS convergence implies ptwise plancherel convergence}]
To deduce the corollary from Theorem \ref{thm:BS convergence implies plancherel convergence} we need to show that every singleton $\{\pi\} \subset \widehat{G}$ is relatively compact and $\nu^G$-regular.

The group $G$ is \emph{liminal}\footnote{A l.c.s.c. group $G$ is type I or is liminal if and only if $\widehat{G}$ is $T_0$ or  $T_1$ in the Fell topology, respectively. In particular being liminal is a stronger property than type I. } in the sense that $\pi(\varphi)$ is a compact operator for every function $\varphi \in C_c^\infty(G)$.  It is shown in \cite{glimm1961type} that the dual space $\widehat{G}$ of a liminal group is $T_1$. In particular every singleton $\{\pi\} \subset \widehat{G}$ is closed and so relatively compact. 

Since $G$ is l.c.s.c. it follows from the definitions that Fell's topology on $\widehat{G}$ is first countable.  Recall that in a $T_1$-space every subset is equal to the intersection of the open sets containing it. Therefore the singleton $\{\pi\}$ is equal to the intersection of a countable family of open sets. Moreover Fell's topology is locally compact  and compact subsets have finite Plancherel measure \cite[18.1.2, 18.8.4]{dixmier1977c}. Therefore $\{\pi\}$ is  $\nu^G$-regular, as required.
\end{proof}

\bibliographystyle{alpha}
\bibliography{qp_betti}

\end{document}